\documentclass[11pt,A4paper]{article}

\usepackage[left=30mm,right=30mm,top=30mm,bottom=30mm]{geometry}
\usepackage{labelfig}
\usepackage{epsfig}
\usepackage{epstopdf}
\usepackage{soul}
\usepackage{xfrac}
\usepackage[T1]{fontenc}
\usepackage[percent]{overpic}

\usepackage{color}
\usepackage{amsthm,amsmath,amssymb}
\usepackage{booktabs}
\usepackage{mathpazo}
\usepackage{microtype}
\usepackage{overpic}
\usepackage{bm}
\usepackage{sectsty}
\usepackage{pinlabel}
\usepackage{tikz}
\usetikzlibrary{matrix,arrows,patterns,intersections,calc,decorations.pathmorphing}
\usepackage[
pdftitle={PDFTitle},
pdfauthor={Hugo Parlier},
ocgcolorlinks,
linkcolor=linkred,
citecolor=linkred,
urlcolor=linkblue]
{hyperref}

\definecolor{linkred}{RGB}{0,191,255} 
\definecolor{linkblue}{RGB}{16, 78, 139}

\usepackage[hang,flushmargin]{footmisc}
\usepackage{enumitem}
\usepackage{titlesec}
\titlespacing{\section}{0pt}{12pt}{0pt}
\titlespacing{\subsection}{0pt}{6pt}{0pt}

\titlelabel{\thetitle.\quad}

\makeatletter 

\long\def\@footnotetext#1{%
	\H@@footnotetext{%
		\ifHy@nesting 
		\hyper@@anchor{\@currentHref}{#1}%
		\else 
		\Hy@raisedlink{\hyper@@anchor{\@currentHref}{\relax}}#1%
		\fi 
}}

\def\@footnotemark{%
	\leavevmode 
	\ifhmode\edef\@x@sf{\the\spacefactor}\nobreak\fi 
	\H@refstepcounter{Hfootnote}%
	\hyper@makecurrent{Hfootnote}%
	\hyper@linkstart{link}{\@currentHref}%
	\@makefnmark 
	\hyper@linkend 
	\ifhmode\spacefactor\@x@sf\fi 
	\relax 
}%

\makeatother 

\theoremstyle{plain}
\newtheorem{theorem}{Theorem}[section]
\newtheorem*{theorem-otal}{Theorem 1.3}
\newtheorem{proposition}[theorem]{Proposition}
\newtheorem{lemma}[theorem]{Lemma}
\newtheorem{corollary}[theorem]{Corollary}

\theoremstyle{definition}

\newtheorem{remark}[theorem]{Remark}

\newcommand{\R}{{\mathbb R}}

\newcommand{\Hyp}{{\mathbb H}}
\newcommand{\N}{{\mathbb N}}

\newcommand{\F}{{\mathbb F}}

\newcommand{\Z}{{\mathbb Z}}
\newcommand{\C}{{\mathbb C}}

\newcommand{\T}{{\mathcal T}}

\newcommand{\FF}{\mathcal F}

\newcommand{\X}{\mathcal{X}}

\newcommand{\Hom}{\mathrm{Hom}}
\newcommand{\hyp}{\mathrm{hyp}}
\newcommand{\irr}{\mathrm{irr}}

\newcommand{\Out}{\mathrm{Out}}

\newcommand{\PSLtwoR}{\mathrm{PSL}(2,\R)}

\newcommand{\SLtwoC}{\mathrm{SL}(2,\C)}

\newcommand{\SLtwoR}{\mathrm{SL}(2,\R)}

\newcommand{\tr}{\mathrm{tr}\,}

\definecolor{light-gray}{gray}{0.95}

\sectionfont{\large \bfseries}
\subsectionfont{\normalsize}

\long\def\symbolfootnote[#1]#2{\begingroup%
	\def\thefootnote{\fnsymbol{footnote}}\footnote[#1]{#2}\endgroup}

\def\blfootnote{\xdef\@thefnmark{}\@footnotetext}

\usepackage{pdfpages} 
\definecolor{darkorange}{rgb}{1.0, 0.55, 0.0}
\definecolor{aogreen}{rgb}{0.0, 0.5, 0.0}
\begin{document}
	
{\Large \bfseries Ordering curves on surfaces}

{\large Hugo Parlier, Hanh Vo\symbolfootnote[1]{Hanh Vo is supported by the AMS-Simons Travel Grant. } and Binbin Xu\symbolfootnote[2]{\small Binbin Xu was supported by the Fundamental Research Funds for the Central Universities, Nankai University (Grant Number 63231055), Natural Science Foundation of Tianjin (Grant number 22JCYBJC00690) and LPMC at Nankai University.
}}
	
\vspace{0.5cm}
	
{\bf Abstract.}
We study the order of lengths of closed geodesics on hyperbolic surfaces. Our first main result is that the order of lengths of curves determine a point in Teichm\"uller space. In an opposite direction, we identify classes of curves whose order never changes, independently of the choice of hyperbolic metric. We use this result to identify short curves with small intersections on pairs of pants.
\section{Introduction}
	
Lengths of closed curves, or more precisely, lengths of marked closed geodesics, are essential tools in the study of hyperbolic surfaces and their related moduli and Teichm\"uller spaces. They are related to a plethora of interesting phenomena and have been studied from various perspectives including their relationship to other topics such as number theory and dynamical systems.
	
A particularly intriguing phenomenon, first discovered by Randol using work of Horowitz, is the existence of length equivalent curves. These are pairs (or more generally tuples) of homotopy classes of closed curves whose corresponding geodesic representatives are always of equal length when the hyperbolic metric on the surface varies. In particular, the existence of such tuples of curves shows that multiplicity in the length spectra of hyperbolic surfaces is unbounded. Various properties of length equivalent curves are studied, but their complete topological characterization is not known yet. 
	
In another direction, McShane and the first author studied in \cite{McShane-Parlier08} the lack of multiplicities in the simple length spectrum. As a corollary they showed that, in the case of closed surfaces and surfaces with given boundary lengths, the length orders of interior simple closed geodesics determine a hyperbolic metric uniquely. By considering closed geodesics with self-intersections as well, we complete this result, removing the condition on boundary lengths and include the case of pairs of pants. 
\begin{theorem}\label{thm:orderdistiguishpoints}
	Let $\Sigma$ be a finite type orientable surface of negative Euler characteristic and $\mathcal{T}(\Sigma)$ be its Teichm\"uller space. The map that sends a point in $\mathcal{T}(\Sigma)$ to its order of interior curve lengths is injective.
\end{theorem} 

The proof of this theorem is based on a projectively injective embedding of the Teichm\"uller space $\T(P)$ of a pair of pants $P$ into $\R^4$, constructed by considering the lengths of four interior curves on $P$. Since any finite type orientable surface $\Sigma$ of negative Euler characteristic can be covered by finitely many embedded pairs of pants, and any hyperbolic metric on $\Sigma$ is determined by its restrictions to these pair of pants, we have the following result (where $\ell_X([\gamma])$ denotes the length of the unique closed geodesic in the homotopy class of $\gamma$ for a choice of hyperbolic metric $X$):

\begin{theorem}\label{thm:generalsurfaceprojectiveinjection}
    Let $\Sigma_{g,n}$ be an oriented surface of genus $g$ with $n>0$ boundary components such that $2-2g-n<0$. There exists a finite collection of interior curves $\gamma_1,...,\gamma_r$ in $\Sigma$ such that the map
        \[
            \begin{aligned}
                \Phi:\T(\Sigma)&\rightarrow \R^r,\\
                        X&\mapsto (\ell_X([\gamma_1]),...,\ell_X([\gamma_r])),
            \end{aligned}   
        \]
    is projectively injective.

Moreover, the self-intersection number of the curves can be chosen to be at most $k$ where
\begin{itemize}
    \item for $\Sigma_{1,1}$, $k=0$,
    \item for $\Sigma_{0,3}$, $k=12$,
    \item for $\Sigma_{0,n}$ with $n>3$, $k=3$,
    \item for $\Sigma_{g,n}$ with $g>1$, $k=2$.
\end{itemize}
\end{theorem}

\medskip

As mentioned previously, length equivalent curves have the feature that, for any choice of hyperbolic metric, their relative orders of lengths never changes. A different type of example comes from Figure-$8$ curves. Given any hyperbolic metric on $\Sigma$, a figure-$8$ geodesic is always longer than the boundary geodesics of the unique embedded pair of pants containing it. Our next results exhibit a third way to find curves whose length order never changes.

For this we consider the presentation of $\pi_1(P)=\langle a,b\rangle$ such that $a$ and $b$ correspond to two boundary curves of $P$ respectively, and such that $ab$ corresponds to the third boundary curve of $P$. 

Any reduced word $w$ of letters in $\{a^{\pm1}, b^{\pm1}\}$ corresponds to a homotopy class $[\alpha]$ of curve $\alpha$ in $P$. Then, given any hyperbolic metric on $P$, We denote by $\ell_X(w)$ the $X$-length of the geodesic representative in $[\alpha]$. Given any words $u$ and $v$, $uv$ will then correspond to the concatenation of the corresponding curves. Our next result is stated as follows.  

\begin{theorem}\label{thm:orderinvariant}
 Let $u$ and $v$ be two cyclically reduced words of letters in $\{a^{\pm1},b^{\pm1}\}$. If $uv$ is cyclically reduced and $u$ starts and ends with different letters, then for all $X \in\mathcal{T}(P)$, we have
		\[
			\ell_X(u)<\ell_X(uv).
		\]
\end{theorem}

As a special case, if we consider the following collection of words
    \[
        \{ab^{-n}\mid n\in\N\},
    \]
then the above theorem tells us that given any hyperbolic metric $X$ on $P$, for any $n\in\N$, we have
    \[
        \ell_X(ab^{-n})<\ell_X(ab^{-(n+1)}).
    \]
This case can be checked directly in another way, since for any geodesic representative of $ab^{-(n+1)}$, there is an embedding of the homotopy class of $ab^{-n}$. Theorem \ref{thm:orderinvariant} provides more subtle examples. For example, let $u=aab^{-1}$ and $v=a^{-1}b$. By our result, for any hyperbolic metric $X$ on $P$, we have $$\ell_X(aab^{-1})<\ell_X(aab^{-1}a^{-1}b).$$
By considering embeddings or immersions of $P$ into $\Sigma$, we can get such pairs in any $\Sigma$ as well.

Another application of Theorem \ref{thm:orderinvariant} is to provide an algorithm to look for shortest closed geodesics among all those that self-intersect at least $k$-times, often called $k$-systoles. The algorithm is a combination between Theorem \ref{thm:orderinvariant} and an algorithm due to Despr\'e-Lazarus \cite{Despre-Lazarus} which is used to compute the self-intersection number of a curve. It is conjectured \cite{MR3380366} that for any $k\in\N^\ast$, the $k$-systole of a pair of pants always self-intersects $k$ times. The above algorithm help us settle the conjecture for small $k$.
\begin{corollary}\label{cor:ksystole}
    Given any hyperbolic metric on $P$, for any $1\leq k\leq 4$, the $k$-systole always self-intersects $k$ times.
\end{corollary}

This paper is organized as follows. We begin with a preliminary section where we introduce notation, and the Fricke type trace equations which we use to character varieties. In Section \ref{sec:projinj}, we find a projectively injective map from Teichm\"uller space into Euclidean space. We then use this in Section \ref{sec:order} to prove our statement on orders of lengths. In Section \ref{sec:neverchange} we find examples of curves whose length order never changes, and prove Corollary \ref{cor:ksystole}.

\section{Preliminaries}\label{sec:prelim}

After introducing notations we will use throughout the paper, we will review facts on character varieties that we will need.

\subsection{Set-up and notation}
Let $\Sigma$ be an oriented topological surface of genus $g$ with $n$ boundary components and such that the Euler characteristic $\chi(\Sigma)=2-2g-n$ is negative. A \textit{curve} on $\Sigma$ is the image of a continuous map from $S^1$ to $\Sigma$ and is said to be \textit{essential} if it is not homotopic to a point or a boundary component. A curve is \textit{peripheral} if it is homotopic to a boundary component. In what follows, curves will always either be essential or peripheral.

Given any curve $\alpha$, by abuse of notation, we use $\alpha$ to denote one of its parametrizations. Then its self-intersection number $i(\alpha)$ is defined to be the following quantity:
	\[
		i(\alpha):=\frac{1}{2}\#\{(s,t)\in S^1\times S^1\mid s\neq t,\,\,\alpha(s)=\alpha(t)\}.
	\]
The homotopy class of $\alpha$ is denoted by $[\alpha]$. We define its self-intersection number to be
	\[
		i([\alpha]):=\min\{i(\alpha')\mid \alpha'\in[\alpha]\}
	\]
A curve is simple if it has no self-intersection and a curve $\alpha$ is called a \textit{$k$-curve} if $i([\alpha])=k$. In particular, a simple curve is a $0$-curve. 

A \textit{hyperbolic structure} on $\Sigma$ is a Riemannian metric structure with constant curvature $-1$. We can think of them as surfaces which are locally isometric to the hyperbolic plane $(\Hyp,\mathrm{d}_\Hyp)$. We include structures such that boundary components are realized by simple closed geodesics or cusps. The \textit{Teichm\"uller space} $\T(\Sigma)$ of $\Sigma$ is the space of isotopy classes of hyperbolic structures on $\Sigma$. Let $X$ be a hyperbolic structure on $\Sigma$. Given any non-peripheral curve $\alpha$ on $\Sigma$, there is a unique geodesic representative in its homotopy class $[\alpha]$. We denote its $X$-length by $\ell_X([\alpha])$. This induces a map
	\[
		\begin{aligned}
			\ell_\alpha:\T(S)&\rightarrow \R\\
			X&\mapsto \ell_{X}([\alpha]).
		\end{aligned}
	\]
which is a \textit{length function}, on $\T(S)$, associated to $[\alpha]$.
	
Given any complete hyperbolic structure $X$ on $\Sigma$, it induces a complete hyperbolic structure $\widetilde{X}$ on the universal cover $\widetilde{\Sigma}$. We then have an isometry $f$ from $(\widetilde{\Sigma},\widetilde{X})$ to $\Hyp^2$, the hyperbolic plane (which will always be endowed with the standard hyperbolic metric $\mathrm{d}_\Hyp)$. The deck transformations of $\pi_1(\Sigma)$ on $\widetilde{\Sigma}$ induce, through $f$, an isomorphism
	\[
		\rho_{X,f}:\pi_1(\Sigma)\rightarrow\PSLtwoR,
	\]
by identifying the orientation preserving isometry group of $\Hyp^2$ with $\PSLtwoR$.

The isometry from $(\widetilde{\Sigma},\widetilde{X})$ to $\Hyp^2$ is not unique. If $f'$ is another such isometry, then $f'\circ f^{-1}$ is an isometry of $\Hyp^2$, hence an element $A$ of $\PSLtwoR$. We then have the following relation
	\[
		\rho_{X,f'}=A\rho_{X,f}A^{-1}.
	\] 
This induces an embedding of $\T(\Sigma)$ into
	\[
		\mathrm{Hom}^{d,f}((\pi_1(\Sigma,\partial\Sigma),\PSLtwoR)/\PSLtwoR,
	\]	
the space of conjugacy classes of discrete faithful representations from $\pi_1(\Sigma,\partial\Sigma)$ to $\PSLtwoR$. Here $\pi_1(\Sigma,\partial\Sigma)$ is the fundamental group $\pi_1(\Sigma)$ with some elements marked as peripheral elements. This is necessary, since different surfaces with boundaries may have isomorphic fundamental groups (for instance pairs of pants and once-punctured tori). 
	
There is one-to-one correspondence between the set of oriented curves on $\Sigma$ and the nontrivial conjugacy classes in $\pi_1(\Sigma)$. Let $\alpha$ be an oriented curve and $a$ be an element of $\pi_1(\Sigma)$ in the conjugacy class associated to $\alpha$. For any hyperbolic metric $X$ on $\Sigma$ and any representation $\rho$ associated to $X$, we have the relation
	\[
		2\cosh \frac{\ell_X([\alpha])}{2}=|\tr\rho(a)|.
	\]
\begin{remark}
	There is a slight, but harmless, abuse in the notation $|\tr\rho(a)|$. This is because it does not make sense to talk about trace of an element in $\PSLtwoR$, since 
		\[
			\PSLtwoR=\SLtwoR/\{\pm I_2\},
		\]
	but we can still talk about the absolute value of the trace.
\end{remark}

\subsection{The $\SLtwoC$-character variety of the rank $2$ free group}	
Let $\F_2$ be the free group on two generators. Throughout, we will let $(u,v)$ be a free basis of $\F_2$, then all non-identity elements in $\F_2$ can be written uniquely as reduced words of letters $\{u^{\pm1},v^{\pm1}\}$. More precisely, a word $w$ of letters $\{u^{\pm1},v^{\pm1}\}$ is \textit{reduced} if any pair of adjacent letters in $w$ are not inverse to each other. If moreover the beginning and ending letters of $w$ are not inverse to each other either, the word $w$ is said to be \textit{cyclically reduced}.

We first recall the following classical result of Fricke and Vogt.
\begin{theorem}[Fricke \cite{Fricke1896}, Vogt \cite{Vogt1889}]\label{thm:frickevogt}
	Let 
		\[
			f:\SLtwoC\times\SLtwoC\rightarrow\C
		\]
	be a regular function which is invariant under the diagonal action of $\SLtwoC$ by conjugation. There exists a polynomial function $F(x,y,z)\in\C[x,y,z]$ such that: for any pair $(M,N)\in\SLtwoC\times\SLtwoC$, we have
	\[
		f(M,N)=F(\tr(M),\tr(N),\tr(MN)).
	\]
	Furthermore, for all $(x,y,z)\in\C^3$, there exists $(M,N)\in\SLtwoC\times\SLtwoC$ such that
	\[
		\begin{bmatrix}
			x\\
			y\\
			z
		\end{bmatrix}=\begin{bmatrix}
			\tr(M)\\
			\tr(N)\\
			\tr(MN)
		\end{bmatrix}.
	\]
	Conversely, if $x^2+y^2+z^2-xyz\neq 4$ and $(M,N)$, $(M',N')\in\SLtwoC\times\SLtwoC$ satisfy
	\[
		\begin{bmatrix}
			x\\
			y\\
			z
		\end{bmatrix}=\begin{bmatrix}
			\tr(M)\\
			\tr(N)\\
			\tr(MN)
		\end{bmatrix}=\begin{bmatrix}
			\tr(M')\\
			\tr(N')\\
			\tr(M'N')
		\end{bmatrix}.
	\]
	Then there exists an element $Q\in\SLtwoC$ such that $(M,N)=(QM'Q^{-1},QN'Q^{-1})$.
\end{theorem}
\begin{remark}\label{rmk:traceconditionforirreducible}
     Given a representation $\eta:\F_2\rightarrow\SLtwoC$, we consider $x=\tr\eta(u)$, $y=\tr\eta(v)$ and $z=\tr\eta(uv)$ (where $u,v$ is our free basis). The condition $x^2+y^2+z^2-xyz\neq 4$ is equivalent to the condition that $\eta(u)$ and $\eta(v)$ share no common eigenvector in $\C^2$, in which case the representation $\eta$ is said to be irreducible. Hence, up to conjugacy, an irreducible representation $\eta$ is determined by the triple $(\tr\eta(u),\tr\eta(v),\tr\eta(uv))$.
\end{remark}

Let $P$ be a thrice-punctured sphere, commonly called a pair of pants.  The fundamental group $\pi_1(P)$ is isomorphic to $\F_2$. Hence, the Teichm\"uller space $\T(P)$ can be considered as part of the character variety
	\[
		\Hom(\F_2,\PSLtwoR)/\PSLtwoR.
	\]
We can moreover lift representations of $\PSLtwoR$ to representations of $\SLtwoR$, where it now makes sense to talk about traces of matrices. More precisely, consider a representation 
        \[
            \rho:\F_2\rightarrow \PSLtwoR.
        \]
Given any basis $(u,v)$ of $\F_2$, let $\rho(u)=A$ and $\rho(v)=B$. By taking a lift $\widetilde{A}$ of $A$ and a lift $\widetilde{B}$ of $B$ in $\SLtwoR$, we can define a representation
	\[
		\widetilde{\rho}:\F_2\rightarrow \SLtwoR,
	\]
with $\widetilde{\rho}(u)=\widetilde{A}$ and $\widetilde{\rho}(v)=\widetilde{B}$. Let $\pi$ be the projection from $\SLtwoR$ to $\PSLtwoR$: we now have $\rho=\pi\circ\widetilde{\rho}$.
A similar discussion can be done for the Teichm\"uller space of one-holed torus whose fundamental group is also isomorphic to $\F_2$.

Using Theorem \ref{thm:frickevogt}, Goldman \cite{Goldman09} studied the character variety
	\[
		\X(\F_2,\SLtwoR):=\mathrm{Hom}(\F_2,\SLtwoR)/\SLtwoR
	\]
in an algebraic way in detail. For a choice $(u,v)$ of free basis of $\F_2$, consider the map
    \[
        \begin{aligned}
            \Phi:\X(\F_2,\SLtwoR)&\rightarrow \R^3,\\
                    [\rho]&\mapsto (\tr\rho(u),\tr\rho(v),\tr\rho(uv)).
        \end{aligned}
    \]
Goldman showed the following theorem.
\begin{theorem}\label{thm:liftspairofpants}
    By identifying $\pi_1(P)$ with $\F_2$, such that $u$, $v$ and $uv$ correspond to the three peripheral curves of $P$, the subset of $\X(\F_2,\SLtwoR)$ corresponding to lifts of representations associated to hyperbolic structures on $P$ can be identified with the following disjoint union of subsets of $\R^3$ under the map $\Phi$:
		\[
		\begin{aligned}
			\FF_P=\{(x,y,z)\in\R^3\mid x\le-2,y\le-2,z\le-2\}&\coprod\{(x,y,z)\in\R^3\mid x\le-2,y\ge2,z\ge2\}\,\\
			&\coprod\{(x,y,z)\in\R^3\mid x\ge2,y\le-2,z\ge2\},\\
			&\coprod\{(x,y,z)\in\R^3\mid x\ge2,y\ge2,z\le-2\}.
		\end{aligned}
		\]
	Moreover there is a $\Z_2\times\Z_2$ action on it corresponding to sign changing, and the Teichm\"uller space $\T(P)$ can be identified with the subset $\FF$ quotient by this action.
\end{theorem}
In particular, we can identify $\T(P)$ to any of the four connected components of $\FF_P$. Without loss of generality, we may consider
    \[
        \Omega=(-\infty,-2]\times(-\infty,-2]\times(-\infty,-2],
    \]
and the identification from $\T(P)$ to $\Omega$ can be given as follows:
    \[
        \begin{aligned}
            \varphi:\T(P)&\rightarrow \Omega\\
                    X &\mapsto \left(-2\cosh\frac{\ell_X(\alpha_1)}{2},-2\cosh\frac{\ell_X(\alpha_2)}{2},-2\cosh\frac{\ell_X(\alpha_3)}{2}\right)
        \end{aligned}
    \]
where $\alpha_1$, $\alpha_2$ and $\alpha_3$ are the three peripheral curves on $P$ with orientation induced by that on $P$. We denote by $\rho_{(x,y,z)}$ the representation determined by $(x,y,z)\in\Omega$. Then we have the embedding of $\T(P)$ into $\X(\F_2,\SLtwoR)$:
    \[
        \begin{aligned}
            \widetilde{\varphi}:\T(P)&\rightarrow \X(\F_2,\SLtwoR),\\
                    X&\mapsto \rho_{\varphi(X )}.
        \end{aligned}
    \]

\section{A projective injection of $\T(\Sigma)$ into a Euclidean space}\label{sec:projinj}
Our first results aim to use the description of the character variety to find a projectively injective embedding of Teichm\"uller space into Euclidean spaces.

We start by studying representations from $\F_2$ to $\SLtwoR$. We observe the following:
\begin{proposition}\label{prop:traceprojectivelyinjection}
Let $M_1$, $M_2$, $N_1$ and $N_2$ be four matrices in $\SLtwoR$, where $M_1$ and $N_1$ are hyperbolic and share no common eigenvector in $\R^2$. Then
        \[
            (\tr M_1,\tr N_1,\tr M_1N_1,\tr M_1^{-1}N_1)=\lambda(\tr M_2,\tr N_2,\tr M_2N_2,\tr M_2^{-1}N_2),
        \]
    for some $\lambda\in\R$, if and only if there exists a matrix $C\in \SLtwoR$, such that $M_2=CM_1C^{-1}$ and $N_2=CN_1C^{-1}$.
\end{proposition}
\begin{proof}
    We consider a trace formula for matrices in $\SLtwoC$, which will be applied to matrices in $\SLtwoR$. More precisely, given any two matrices $M_1,N_1\in\SLtwoR$, we have
    \[
        \tr M_1\tr N_1=\tr M_1N_1+\tr M_1^{-1}N_1.
    \]
    Given any matrices $M_2,N_2\in\SLtwoR$ and for any $\lambda\in \R^*$, if 
    \[
            (\tr M_1,\tr N_1,\tr M_1N_1,\tr M_1^{-1}N_1)=\lambda(\tr M_2,\tr N_2,\tr M_2N_2,\tr M_2^{-1}N_2),
    \]
    then we have
    \[
        \lambda^2\tr M_2\tr N_2=\lambda\tr M_2N_2+\lambda\tr M_2^{-1}N_2=\lambda\tr M_2\tr N_2.
    \]
    Therefore we have
    \[
        \lambda(\lambda-1)\tr M_2\tr N_2=0.
    \]

    If $M_1$, $N_1$, $M_2$ and $N_2$ are hyperbolic, then both $\lambda$ and $\tr M_2\tr N_2$ are different from $0$, hence $\lambda=1$.
    
    Since $M_1$ and $N_1$ share no common eigenvector in $\R^2$, by the discussion in Remark \ref{rmk:traceconditionforirreducible}, we have
        \[
            (\tr M_1)^2+(\tr N_1)^2+(\tr M_1\tr N_1)^2-\tr M_1\tr N_1\tr M_1N_1\neq 4.
        \]
    By the result of Fricke and Vogt (Theorem \ref{thm:frickevogt}), we have a matrix $C\in\SLtwoR$, such that $M_2=CM_1C^{-1}$ and $N_2=CN_1C^{-1}$.

    The other direction is a consequence of the fact that the trace of a matric is preserved under conjugation. 
\end{proof}

Consider the basis $(u,v)$ of $\F_2$, and the subset of $\X(\F_2,\SLtwoR)$ consisting of all irreducible representations sending $u$, $v$, $uv$ and $u^{-1}v$ to hyperbolic elements: we denote it by $\X_{\hyp}^{\irr}(u,v)$.

The condition on the image of $u$ and $v$ tells us that $\X_{\hyp}^{\irr}(u,v)$ is an open subset of $\X(\F_2,\SLtwoR)$. On the other hand, this subset is not $\Out(\F_2)$-invariant. For example, a representation associated to a pair of pants with one boundary realized as a cone point singularity is such an example. 
\begin{remark}
    The subset $\X_{\hyp}^{\irr}(u,v)$ of $\X(\F_2,\SLtwoR)$ is not $\Out(\F_2)$-invariant. One may consider a point $[\rho]\in\X(\F_2,\SLtwoR)$ associated to the hyperbolic thrice-punctured sphere. If we identify $\F_2$ with $\pi_1(P)$ in the way that one of $u$, $v$, $uv$ and $u^{-1}v$ corresponds to a boundary component of $P$, then its $\rho$-image is a parabolic element. Hence, $[\rho]$ is not in $\X_{\hyp}^{\irr}(u,v)$. On the other hand, we can identify $\F_2$ with $\pi_1(P)$ such that $u$, $v$, $uv$ and $u^{-1}v$ all correspond to interior curves, and then $[\rho]\in \X_{\hyp}^{\irr}(u,v)$.
\end{remark}

As a corollary of the above proposition, we have:
\begin{corollary}
    The continuous map
        \[
            \begin{aligned}
                \Psi: \X_{\hyp}^{\irr}(u,v)&\rightarrow \R^4\\
                        [\rho]&\mapsto (\tr\rho(u),\tr\rho(v),\tr\rho(uv),\tr\rho(u^{-1}v))
            \end{aligned}
        \]
    is projectively injective.
\end{corollary}

Given any hyperbolic element $M\in \SLtwoR$, it acts on $\Hyp^2$ by a fractional linear map and has positive translation distance
    \[
        d(M):=\inf\{\mathrm{d}_\Hyp(p, M(p))\mid p\in\Hyp \}>0.
    \]
As usual, we also have
    \[
        |\tr M|=2\cosh\frac{d(M)}{2}.
    \]
Next we would like to prove the following result.
\begin{proposition}\label{prop:projectiveinjectionalgebraic}
    The continuous map
        \[
            \begin{aligned}
                \widetilde{\Psi}: \X_{hyp}^{irr}(u,v)&\rightarrow \R^4\\
                        [\rho]&\mapsto (d(\tr\rho(u)),d(\tr\rho(v)),d(\tr\rho(uv)),d(\tr\rho(u^{-1}v)))
            \end{aligned}
        \]
    is projectively injective.
\end{proposition}
To prove this proposition, we start by showing two lemmas.    
\begin{lemma}\label{lem:techlem}
For any real numbers $x$ and $y$ with $x>y\ge0$, the function
    \[
        \frac{\cosh tx}{\cosh ty}
    \]
is strictly increasing for $t\in\R_{>0}$.
\end{lemma}
\begin{proof}
    We compute the derivative of the following function
        \[
            \frac{\cosh tx}{\cosh ty}
        \]
    of variable $t$, and get
        \[
            \begin{aligned}
                \left(\frac{\cosh tx}{\cosh ty}\right)'&=\frac{x\sinh tx\cosh ty-x\cosh tx\sinh ty+x\cosh tx\sinh ty-y\cosh tx\sinh ty}{\cosh^2 ty}\\
                &=\frac{x\sinh t(x-y)+(x-y)\cosh tx\sinh ty}{\cosh^2 ty}>0
            \end{aligned}
        \]
    for all $t\in\R_{>0}$. Hence the lemma.    
\end{proof}
We borrow the second technical lemma from \cite{McShane-Parlier08}:
 \begin{lemma}\label{lem:projectiveinjectiontechlem}
        	Let $x_1,x_2,y_1,\dots,y_n$ be positive numbers satisfying $x_1>y_k$ for all $k\in\{1,\dots,n\}$. Then the following real function defined for $t\in\mathbb{R}_{>0}$
        		\[
        			\cosh(x_1t)+\cosh(x_2t)-\sum_{k=1}^{n}\cosh(y_kt).
        		\]
            has at most one strictly positive zero.
\end{lemma}

We denote by $\epsilon_u,\epsilon_v,\epsilon_{uv},\epsilon_{u^{-1}v}\in\{1,-1\}$ the signs of $\tr\rho(u)$, $\tr\rho(v)$, $\tr\rho(uv)$ and $\tr\rho(u^{-1}v)$, and we have
    \[
        2\epsilon_u\cosh\frac{d(\tr\rho(u))}{2}\epsilon_v\cosh\frac{d(\tr\rho(v))}{2}=\epsilon_{uv}\cosh\frac{d(\tr\rho(uv))}{2}+\epsilon_{u^{-1}v}\cosh\frac{d(\tr\rho(u^{-1}v))}{2}.
    \]
By standard trigonometric formulas, we can deduce
    \[
        \begin{aligned}
            &\epsilon_u\epsilon_v\left(\cosh\frac{d(\tr\rho(u))+d(\tr\rho(v))}{2}+\cosh\frac{d(\tr\rho(u))-d(\tr\rho(v))}{2}\right)\\
            =&\epsilon_{uv}\cosh\frac{d(\tr\rho(uv))}{2}+\epsilon_{u^{-1}v}\cosh\frac{d(\tr\rho(u^{-1}v))}{2}.
        \end{aligned}
    \]

We now observe the following:

\begin{proposition}
    The equation
        \[
        \begin{aligned}
            &\epsilon_u\epsilon_v\left(\cosh t\frac{d(\tr\rho(u))+d(\tr\rho(v))}{2}+\cosh t\frac{d(\tr\rho(u))-d(\tr\rho(v))}{2}\right)\\
            =&\epsilon_{uv}\cosh t\frac{d(\tr\rho(uv))}{2}+\epsilon_{u^{-1}v}\cosh t\frac{d(\tr\rho(u^{-1}v))}{2}
        \end{aligned}
    \]
    has a unique solution $t=1$ in $\R_{>0}$.
\end{proposition}
\begin{proof}
    Without loss of generality, we may assume that $d(\tr\rho(u))-d(\tr\rho(v))\ge 0$. To simply notation, we consider the following setup
        \[
            \cosh t\ell_1+\cosh t\ell_2=\epsilon\cosh t\ell_3+\epsilon'\cosh t\ell_4,
        \]
    with  $\ell_1>\ell_2\ge 0$, $\ell_3\ge \ell_4>0$ and $\epsilon,\epsilon'\in\{1,-1\}$ which we analyze case by case.

    Notice that $\epsilon=\epsilon'=-1$ is impossible. If $\epsilon=1$ and $\epsilon'=-1$, we have
        \[
            \cosh t\ell_1+\cosh t\ell_2+\cosh t\ell_4=\cosh t\ell_3.
        \]
    If $\ell_1\ge \ell_3$, then we consider
        \[
            \frac{\cosh t\ell_1}{\cosh t\ell_3}+\frac{\cosh t\ell_2}{\cosh t\ell_3}+\frac{\cosh t\ell_4}{\cosh t\ell_3}=1.
        \]
    By Lemma \ref{lem:techlem}, we have
        \[
            \frac{\cosh t\ell_1}{\cosh t\ell_3}\ge\lim_{t\rightarrow 0}\frac{\cosh t\ell_1}{\cosh t\ell_3}=1.
        \]
    Since $\frac{\cosh t\ell_2}{\cosh t\ell_3}+\frac{\cosh t\ell_4}{\cosh t\ell_3}$ is positive for all $t>0$, we have 
        \[
            \frac{\cosh t\ell_1}{\cosh t\ell_3}+\frac{\cosh t\ell_2}{\cosh t\ell_3}+\frac{\cosh t\ell_4}{\cosh t\ell_3}>1,
        \]
    and hence this case is impossible. 

    If $\ell_1\le \ell_3$, we come back to 
        \[
            \frac{\cosh t\ell_1}{\cosh t\ell_3}+\frac{\cosh t\ell_2}{\cosh t\ell_3}+\frac{\cosh t\ell_4}{\cosh t\ell_3}=1.
        \]
    By Lemma \ref{lem:techlem}, the first term and the last term on the left hand side are both decreasing, and the second term is strictly decreasing. Hence, equality holds for at most one $t$.
    
    If $\epsilon=-1$ and $\epsilon'=1$, we have
        \[
            \cosh t\ell_1+\cosh t\ell_2+\cosh t\ell_3=\cosh t\ell_4,
        \]
    and since $\cosh t\ell_3\ge \cosh t\ell_4$, and the first two terms are both positive, this case is impossible.
    
   If $\epsilon=1$ and $\epsilon'=1$, we have
        \[
            \cosh t\ell_1+\cosh t\ell_2=\cosh t\ell_3+\cosh t\ell_4.
        \]
    Using Lemma \ref{lem:projectiveinjectiontechlem}, we have the result in this case.    
\end{proof}

The space $\X_{\hyp}^{\irr}(u,v)$ contains copies of the following spaces:
\begin{itemize}
    \item the Teichm\"uller space $\T(P)$ of a pair of pants;
    \item the Teichm\"uller space $\T(\Sigma_{1,1})$ of one-holed torus;
    \item the Teichm\"uller space $\T^p(\Sigma_{1,1})$ of once-punctured torus;
    \item the Teichm\"uller space $\T^c(\Sigma_{1,1})$ of a torus with a cone-point.
\end{itemize}
The map $\widetilde{\Psi}$ in Proposition \ref{prop:projectiveinjectionalgebraic} induces projective injections from any of the above spaces to $\R^4$. We will describe precisely this projective injection for $\T(P)$ in the following.

Recall the map 
    \[
        \begin{aligned}
            \widetilde{\varphi}:\T(P)&\rightarrow \X(\F_2,\SLtwoR)\\
                                X &\mapsto [\rho_X]
        \end{aligned}
    \]
Let $(a,b)$ be the basis of $\F_2$, such that $a$, $b$ and $ab$ correspond to the three boundary components $\alpha_1$, $\alpha_2$ and $\alpha_3$. Assume that the basis $(u,v)$ of $\F_2$ is such that no elements among $u$, $v$, $uv$ and $u^{-1}v$ are conjugate to any of $a^{\pm1}$, $b^{\pm1}$ and $(ab)^{\pm1}$. Let $[\beta_1]$, $[\beta_2]$, $[\beta_3]$ and $[\beta_4]$ be the four homotopy classes of curves on $P$ associated to $u$, $v$, $uv$ and $u^{-1}v$ respectively. 
\begin{corollary}\label{cor:pairofpantsprojectiveinjection}
    The map
		\[
			\begin{aligned}
				\widetilde{\Psi}\circ \widetilde{\varphi}: \T(P)&\rightarrow \R^4\\
					X &\mapsto \left(\ell_{X}([\beta_1]),\ell_{X}([\beta_2]),\ell_X([\beta_3]),\ell_X([\beta_4])\right)
			\end{aligned}
		\]
    is projectively injective.
\end{corollary}
\begin{remark}
    Given a metric $X$, when a boundary component is not a puncture, we can also choose the corresponding boundary curve to build the above map. We will use this fact for a later discussion on general surfaces.  
\end{remark}

For later use, we also give the corollary for $\T(\Sigma_{1,1})$ case. We remark that the case of $\T(\Sigma_{1,1})$ has been dealt in \cite{McShane-Parlier08}. More precisely, we identify $\pi_1(\Sigma_{1,1})$ with $\F_2$ in either way. Let $[\delta_1], [\delta_2], [\delta_3], [\delta_4]$ be the four homotopy classes of curves in $\Sigma_{1,1}$ associated to $u$, $v$, $uv$ and $u^{-1}v$ respectively.
\begin{corollary}[McShane-Parlier]\label{cor:oneholedtorusprojectiveinjection}
    The map
		\[
			\begin{aligned}
				\T(\Sigma_{1,1})&\rightarrow \R^4\\
					X &\mapsto \left(\ell_{X}([\delta_1]),\ell_{X}([\delta_2]),\ell_X([\delta_3]),\ell_X([\delta_4])\right)
			\end{aligned}
		\]
    is projectively injective.
\end{corollary}

We end this section with the proof of Theorem \ref{thm:generalsurfaceprojectiveinjection}.

\begin{proof}[Proof of Theorem \ref{thm:generalsurfaceprojectiveinjection}]
 The result for $\Sigma_{1,1}$ and $\Sigma_{g,0}$ ($g>1$) can be found in \cite{McShane-Parlier08}. We briefly review the proof. By identifying $\pi_1(\Sigma_{1,1})$ with $\F_2$, all elements of $\pi_1(\Sigma_{1,1})$ that can be part of free bases correspond to homotopy classes of simple closed interior curves. Then, by Corollary \ref{cor:oneholedtorusprojectiveinjection}, the theorem holds for $\Sigma_{1,1}$.
    
 Now consider $\Sigma_{g,0}$. By adding simple curves transversal to a pair of pants decomposition, one may find a collection of curves $\{\gamma_1,...,\gamma_k\}$, whose lengths determine hyperbolic metrics on $\Sigma_{g,0}$. In each pair of pants decomposition of $\Sigma_{g,0}$, there a non-separating curve, that is that a curve whose complement in $\Sigma_{g,0}$ is connected. Therefore, the curve is contained in an embedded one-holed torus in $\Sigma_{g,0}$. Without loss of generality, we may assume that the curve is $\gamma_k$ and denote by $T$ a one-holed torus containing it. We can then add $\gamma_{k+1}$, $\gamma_{k+2}$ and $\gamma_{k+3}$, three more simple curves in $\Sigma_{g,0}$, such that the lengths of $\gamma_k$, $\gamma_{k+1}$, $\gamma_{k+2}$ and $\gamma_{k+3}$ induce a projectively injective continuous map from $\T(T)$ into $\R^4$. By considering the lengths of $\gamma_1$,...,$\gamma_{k+3}$, we have a projectively injective continuous map from $\T(\Sigma_{g,0})$ into $\R^{k+3}$. Hence, we prove Theorem \ref{thm:generalsurfaceprojectiveinjection} in this case.
    
    For $\Sigma_{0,3}=P$, we are not allowed to chose the peripheral curves so we proceed differently. By Corollary \ref{cor:pairofpantsprojectiveinjection}, using the same setup, we consider the free basis $(u,v)=(ab^{-1}ab^{-2},ab^{-2})$. The four required curves can be chosen to be associated to $ab^{-1}$, $ab^{-2}$, $ab^{-1}ab^{-2}$ and $(ab)^{-2}ab^{-2}$ respectively. This proves Theorem \ref{thm:generalsurfaceprojectiveinjection} for $\Sigma_{0,3}$.

    For a surface $\Sigma_{g,n}$ with more complicated topology, we consider finitely many pants of $\Sigma$, such that the union of their interiors cover the surface $\Sigma_{g,n}$. Notice that the metric information on these pairs of pants will determine the metric information on $\Sigma_{g,n}$. Hence, all we have to do is to apply Corollary \ref{cor:pairofpantsprojectiveinjection} to each pair of pants by choosing four curves carefully for each one.

    When the genus $g$ is greater than $0$, we can chose the pants decompositions such that each pair of pants in this decomposition has at most one boundary component which belongs to $\partial\Sigma$. Using the same setup as for Corollary \ref{cor:pairofpantsprojectiveinjection}, for a pair of pants whose boundary curves are all interior, we can choose four curves associated to $a$, $b$, $ab$ and $ab^{-1}$ respectively, which all have at most $1$ self-intersection. For a pair of pants with exactly one boundary curve peripheral in $\Sigma$, we can choose four curves with at most $2$ self-intersections, which correspond to $a$, $a^{-1}b$, $a^{-2}b$ and $b$. Here we assume that the boundary curves of the pair of pants associated to $a$ and $b$ respectively are both interior. This completes the proof of Theorem \ref{thm:generalsurfaceprojectiveinjection} for $\Sigma_{g,n}$ with $g>0$.

    When $g=0$, we can chose the pants decompositions such that each pair of pants in this decomposition has at most two boundary components on $\partial\Sigma_{0,n}$. (For pairs of pants with at most one boundary component peripheral in $\Sigma_{0,n}$, the discussion will be identical.) Given a pair of pants with exactly two boundary curve peripheral in $\Sigma_{0,n}$, we can choose four curves with at most $3$ self-intersections, corresponding to $a$, $a^{-2}b$, $a^{-3}b$ and $a^{-1}b$. Here we assume that the boundary curve of the pair of pants associated to $a$ is interior. This proves Theorem \ref{thm:generalsurfaceprojectiveinjection} for $\Sigma_{0,n}$.
\end{proof}

\section{Length orders determines points in Teichm\"uller space}\label{sec:order}
In this part, we prove Theorem \ref{thm:orderdistiguishpoints}: we need to show that, given a pair of distinct surfaces in $\T(\Sigma)$, there exist two interior curves on $\Sigma$ whose length order differs on the two surfaces.

More precisely, let $X$ and $X'$ be two points in $\T(\Sigma)$. We use the setup from Theorem \ref{thm:generalsurfaceprojectiveinjection} and consider the projectively injective map $F$. Hence, there are two distinct curves among $[\gamma_1],...,[\gamma_t]$ whose length ratios for $X$ and $X'$ are different. Up to reordering, we can assume that the two curves are $[\gamma_1]$ and $[\gamma_2]$. Hence, we have
	\[
	   \frac{\ell_X([\gamma_1])}{\ell_X([\gamma_2])}\neq \frac{\ell_{X'}([\gamma_1])}{\ell_{X'}([\gamma_2])}.
	\]
Without loss of generality, we may assume that
	\[
	   \frac{\ell_X([\gamma_1])}{\ell_X([\gamma_2])}> \frac{\ell_{X'}([\gamma_1])}{\ell_{X'}([\gamma_2])}.
	\]

Let $\beta$ be a homotopically non trivial oriented curve, such that for any $n\in\Z$, we have $[\beta]\neq [\gamma_1^n]$ and $[\beta]\neq [\gamma_2^n]$. Consider the following two collection of curves
	\[
		\begin{aligned}
			&\{[\beta\ast\gamma_1^j]\mid j\in\N\},\\
			&\{[\beta\ast\gamma_2^k]\mid k\in\N\}.
		\end{aligned}
	\]
It is straightforward to see that for any $X_0\in \T(P)$, we have
	\[
		\begin{aligned}
			\lim_{j\rightarrow\infty}\frac{\ell_{X_0}([\beta\ast\gamma_1^j])}{j}&=\ell_{X_0}([\gamma_1]),\\
			\lim_{k\rightarrow\infty}\frac{\ell_{X_0}([\beta\ast\gamma_2^k])}{k}&=\ell_{X_0}([\gamma_2]).
		\end{aligned}
	\]
Therefore, for any constant $\epsilon>0$, there is an index $K(\epsilon)\in\N$, such that for any $j,k\ge K(\epsilon)$, we have
	\[
		\begin{aligned}
			&\frac{|\ell_{X}([\beta\ast\gamma_1^j])-j\ell_{X}([\gamma_1])|}{j}<\epsilon,&&\frac{|\ell_{X'}([\beta\ast\gamma_1^j])-j\ell_{X'}([\gamma_1])|}{j}&<\epsilon,\\
			&\frac{|\ell_{X}([\beta\ast\gamma_2^k])-k\ell_{X}([\gamma_2])|}{k}<\epsilon,&&\frac{|\ell_{X'}([\beta\ast\gamma_2^k])-k\ell_{X'}([\gamma_2])|}{k}&<\epsilon.
		\end{aligned}
	\]
From the above inequality, we have
	\[
		\left\{\begin{aligned}
			\ell_{X}([\beta\ast\gamma_1^j])>j(\ell_{X}([\gamma_1])-\epsilon)\\
			\ell_{X}([\beta\ast\gamma_2^k])<k(\ell_{X}([\gamma_2])+\epsilon)\\
		\end{aligned}\right.\quad\quad
		\left\{\begin{aligned}
			\ell_{X'}([\beta\ast\gamma_1^j])<j(\ell_{X'}([\gamma_1])+\epsilon)\\
			\ell_{X'}([\beta\ast\gamma_2^k])>k(\ell_{X'}([\gamma_2])-\epsilon)\\
		\end{aligned}\right.
	\]
Then we take $\epsilon>0$ such that
	\[
		\frac{\ell_{X}([\gamma_1])-\epsilon}{\ell_{X}([\gamma_2])+\epsilon}>\frac{\ell_{X'}([\gamma_1])+\epsilon}{\ell_{X'}([\gamma_2])-\epsilon}
	\]
Let $j,k>K(\epsilon)$ be such that
	\[
		\frac{\ell_{X}([\gamma_1])-\epsilon}{\ell_{X}([\gamma_2])+\epsilon}>\frac{k}{j}>\frac{\ell_{X'}([\gamma_1])+\epsilon}{\ell_{X'}([\gamma_2])-\epsilon}
	\]
Then we have
	\[
	\left\{\begin{aligned}
		\ell_{X}([\beta\ast\gamma_1^j])&>\ell_{X}([\beta\ast\gamma_2^k])\\
		\ell_{X'}([\beta\ast\gamma_1^j])&<\ell_{X'}([\beta\ast\gamma_2^k])
	\end{aligned}\right.
	\]
Hence, there is a pair $([\beta\ast\gamma_1^j],[\beta\ast\gamma_2^k])$ of curves whose order of length for $X$ differs from that of $X'$. This ends the proof of Theorem \ref{thm:orderdistiguishpoints}.

\section{Some orders never change}\label{sec:neverchange}
In this part, we will investigate curves in $\Sigma$ whose length order never changes. We start by constructing curves on a pair of pants $P$ with this property.

\subsection{Combinatorial description of curves on $P$}
We first give a combinatorial description of curves on $P$. 

We begin by fixing a bit of terminology. A \textit{path} $\xi$ on $P$ is a continuous map from $[0,1]$ to $P$. A path is said to be \textit{simple} if it is injective. If a path $\xi$ has $\xi(0)$ and $\xi(1)$ on $\partial P$, we call it an \textit{arc}. In the following, we allow ourselves to use ``path" or ``arc" as either the map or its image.

Let $\xi_1$ (resp. $\xi_2$, $\xi_3$) be an arc connecting $\alpha_1$ and $\alpha_3$ (resp. $\alpha_1$ and $\alpha_2$, $\alpha_2$ and $\alpha_3$). This is illustrated in Figure \ref{fig:P1}. Assume moreover that $\xi_1$, $\xi_2$ and $\xi_3$ are simple and pairwise disjoint. By cutting along these arcs, we obtain two hexagons, denoted by $H_+$ and $H_-$. Any curve $\eta$ in $P$ will be cut into paths $e_1,...,e_n$, each one of which is either in $H_+$ or $H_-$.
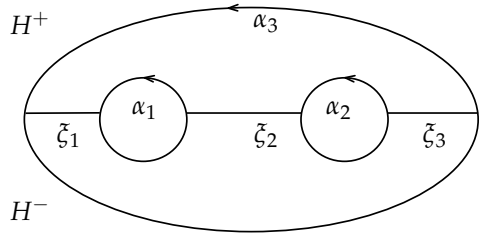
\begin{figure}[ht!] 
\labellist
\small\hair 2pt  
\pinlabel $H^-$ at 31 65
\pinlabel $H^+$ at 31 395
\pinlabel $\xi_1$ at 100 195
\pinlabel $\xi_2$ at 451 195
\pinlabel $\xi_3$ at 751 195

\pinlabel $\alpha_1$ at 235 240
\pinlabel $\alpha_2$ at 580 240
\pinlabel $\alpha_3$ at 451 395
\endlabellist
\centering 
\includegraphics[width=0.4\linewidth]{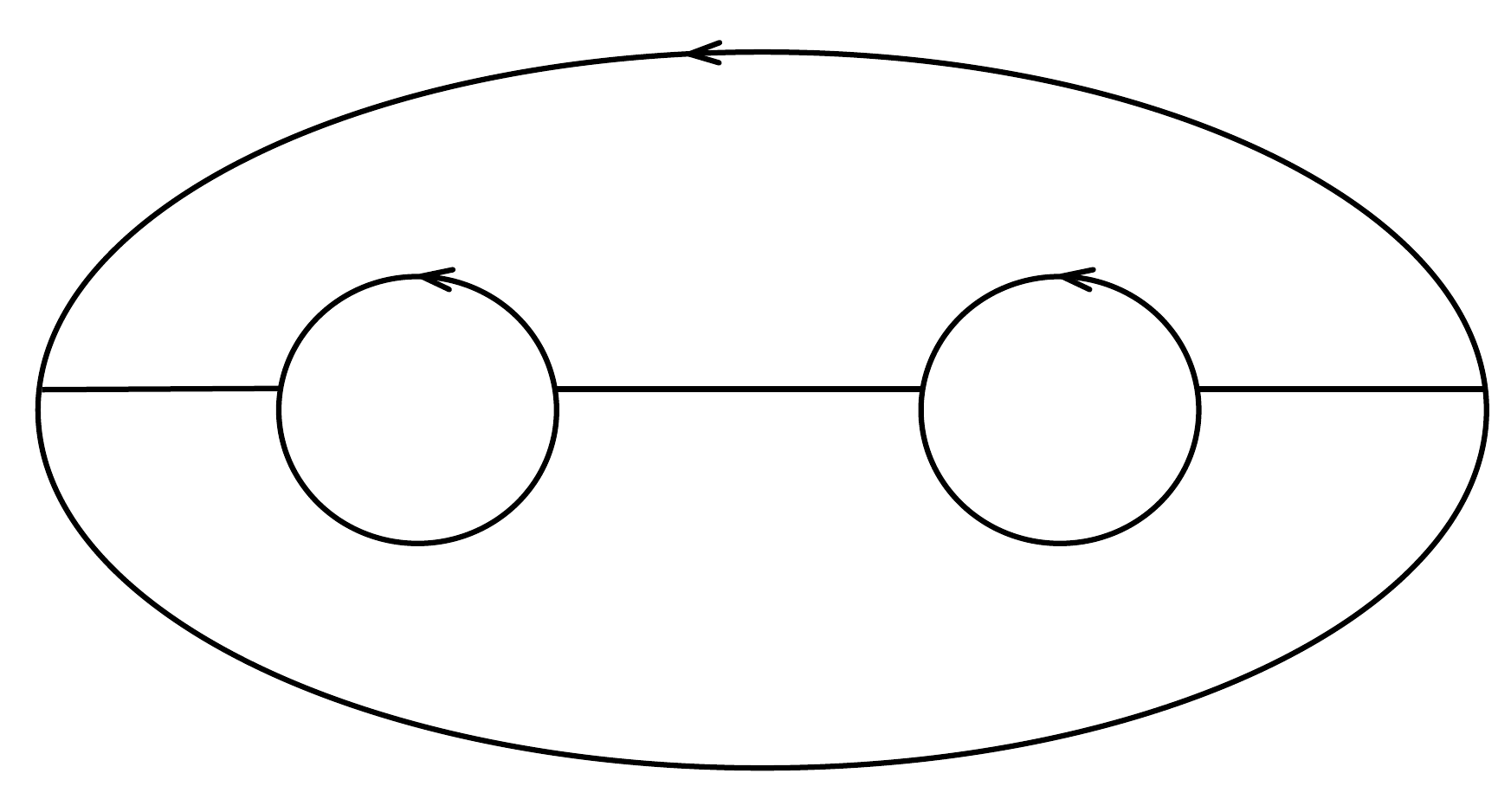}
\caption{Arcs on $P$}\label{fig:P1}
\end{figure}

Conversely, the information on all paths $e_j$ will determine the curve $\eta$. If a path $e_j$ starts on $\xi_i$ and goes into $H_+$ (resp. $H_-$), we denote it by $\xi_i^+$, resp. $\xi_i^-$, see Figure \ref{fig:P2}. In this way, we can associate to the curve $\eta$ a finite word $t_1\cdots t_n$ of letters in $\{\xi_1^{\pm},\xi_2^{\pm},\xi_3^{\pm}\}$.
\begin{figure}[ht!] 
\labellist
\small\hair 2pt  
\pinlabel $H^-$ at 31 65
\pinlabel $H^+$ at 31 395
\pinlabel $\xi_1^-$ at 100 195
\pinlabel $\xi_2^-$ at 451 195
\pinlabel $\xi_3^-$ at 751 195

\pinlabel $\xi_1^+$ at 100 275
\pinlabel $\xi_2^+$ at 451 275
\pinlabel $\xi_3^+$ at 751 275

\pinlabel $\alpha_1$ at 235 240
\pinlabel $\alpha_2$ at 580 240
\pinlabel $\alpha_3$ at 451 395
\endlabellist
\centering 
\includegraphics[width=0.4\linewidth]{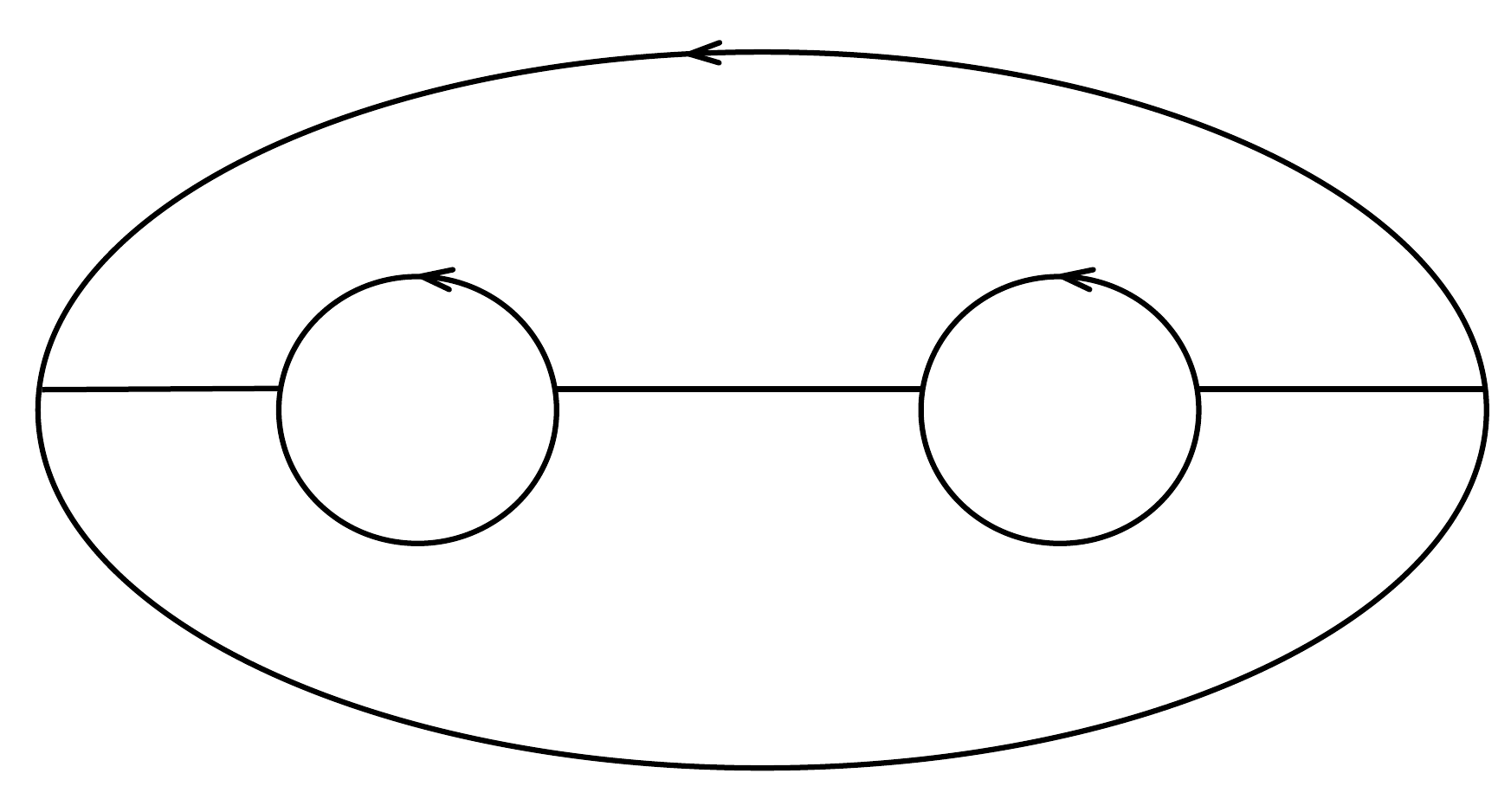}
\caption{The plus and minus versions of $\xi_i$}\label{fig:P2}
\end{figure} 

Let 
    \[
        w=t_{11}\cdots t_{1n_1}t_{21}\cdots t_{2n_2}\cdots t_{k1}\cdots t_{kn_k}
    \]
be a word of letters in $\{\xi_1^{\pm},\xi_2^{\pm},\xi_3^{\pm}\}$, such that for any $j\in\{1,...,k\}$, all letters $t_{j1},...,t_{jn_{j}}$ are in $\{\xi_1^{+},\xi_2^{+},\xi_3^{+}\}$ or $\{\xi_1^{-},\xi_2^{-},\xi_3^{-}\}$ at the same time. Since both $H_+$ and $H_-$ are topological disks, the word
    \[
        w'=t_{11}t_{21}\cdots t_{k1}
    \]
induces a curve homotopic to one induced by $w$. A cyclic permutation on a word will yield another word which induces the same curve on $P$. Moreover, if $\{t_{11},t_{21}\}=\{\xi_i^{+},\xi_i^{-}\}$ for some $i\in\{1,2,3\}$, then the word $w''=t_{31}\cdots t_{k1}$ induces the same curve as $w$ as well.

\begin{figure}[h] 
\labellist
\small\hair 2pt  
\pinlabel $H^-$ at 5 295
\pinlabel $H^+$ at 5 425
\pinlabel $\alpha_1$ at 86 365
\pinlabel $\alpha_2$ at 200 365
\pinlabel $\alpha_3$ at 100 275
\pinlabel $\xi_1$ at 31 350
\pinlabel $\xi_2$ at 130 350
\pinlabel $\xi_3$ at 238 350

\pinlabel $\alpha_1$ at 425 365
\pinlabel $\alpha_2$ at 540 365
\pinlabel $\alpha_3$ at 451 275
\pinlabel $\xi_1$ at 381 350
\pinlabel $\xi_2$ at 480 350
\pinlabel $\xi_3$ at 588 350

\pinlabel $\alpha_1$ at 86 123
\pinlabel $\alpha_2$ at 200 123
\pinlabel $\alpha_3$ at 100 220
\pinlabel $\xi_1$ at 31 107
\pinlabel $\xi_2$ at 133 107
\pinlabel $\xi_3$ at 238 107

\pinlabel $\alpha_1$ at 425 123
\pinlabel $\alpha_2$ at 540 123
\pinlabel $\alpha_3$ at 451 220
\pinlabel $\xi_1$ at 376 107
\pinlabel $\xi_2$ at 490 107
\pinlabel $\xi_3$ at 588 107

\endlabellist
\centering 
\includegraphics[width=12cm]{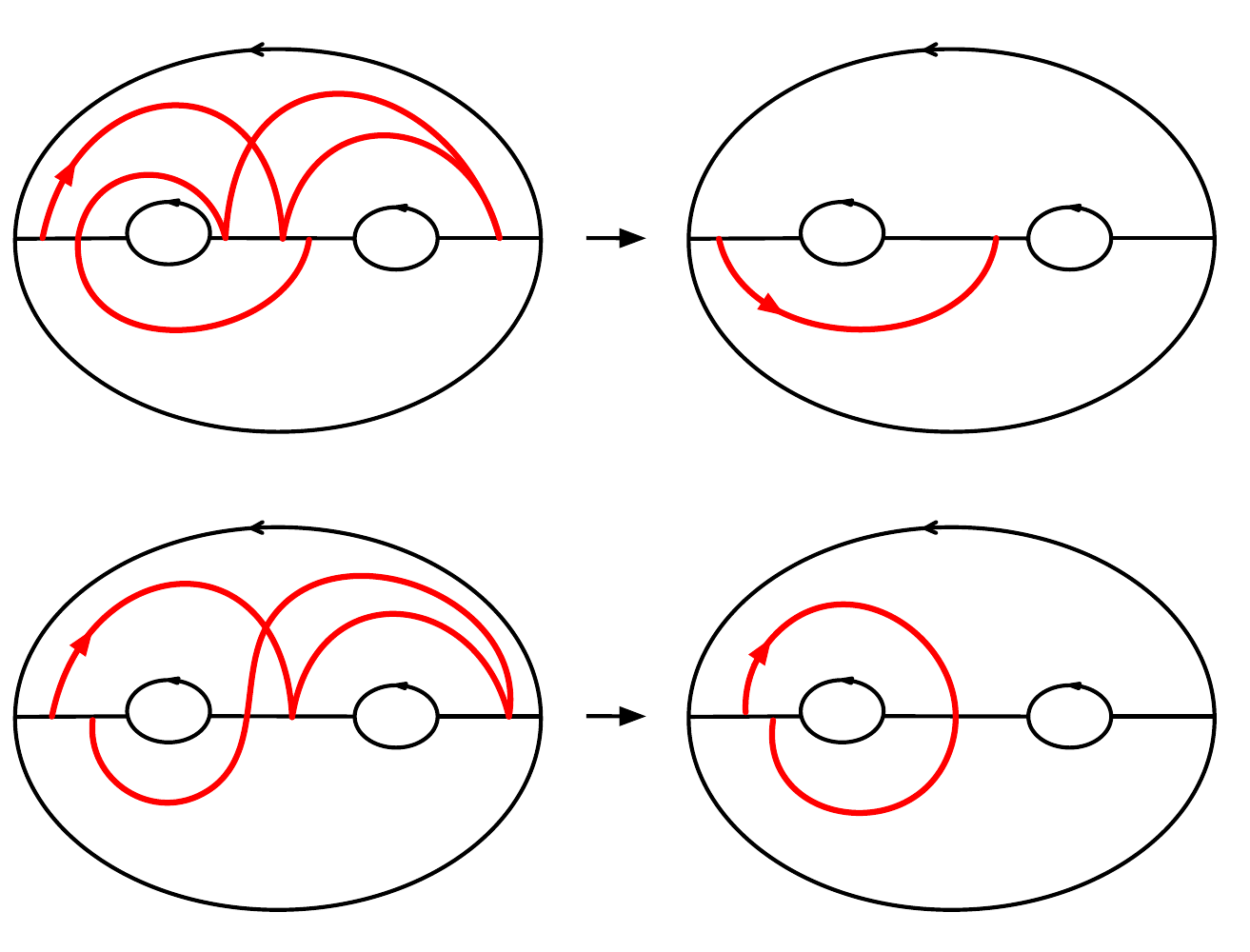}
\caption{Visualizing the homotopies}
\end{figure}
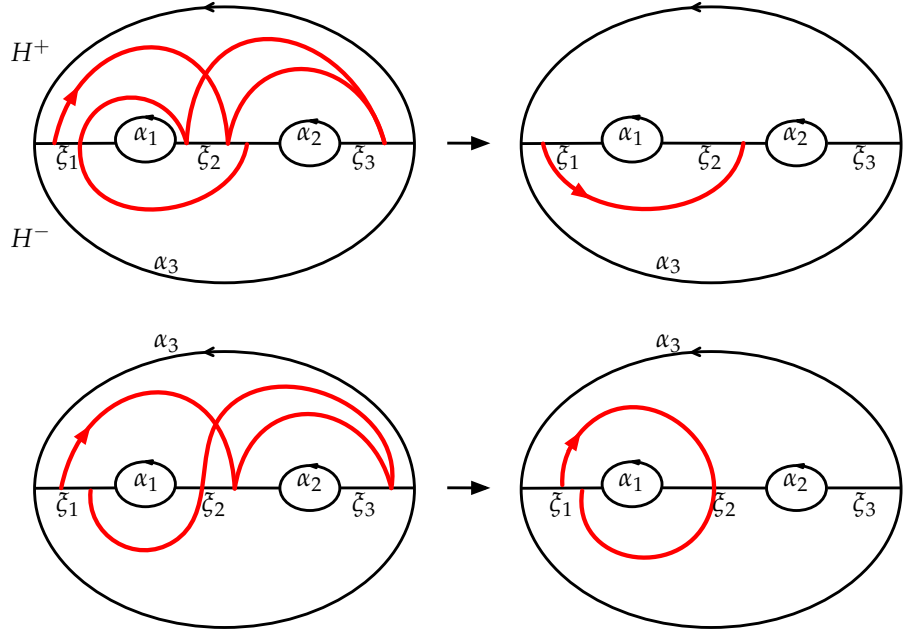 

Therefore, we make some assumptions on the word $w=t_1\cdots t_n$ for the curve $\eta$:
    \begin{itemize}
        \item the number $n$ is even;
        \item if $\eta$ enters $H_+$ (resp. $H_-$) first, then $t_j$ with $j$ odd (resp. even) will be in $\{\xi_1^+,\xi_2^+,\xi_3^+\}$, and $t_j$ with $j$ even (resp. odd) will be in $\{\xi_1^-,\xi_2^-,\xi_3^-\}$;
        \item for any $j\in\{1,...,n-1\}$, we have $\{t_{j},t_{j+1}\}\neq\{\xi_i^{+},\xi_i^{-}\}$ for any $i\in\{1,2,3\}$.
    \end{itemize}
We call words of letters in $\{\xi_1^{\pm},\xi_2^{\pm},\xi_3^{\pm}\}$ with these properties \textit{admissible}. As in the study of elements in $\F_2$, if, moreover, $\{t_{n},t_{1}\}\neq\{\xi_i^{+},\xi_i^{-}\}$ for any $i\in\{1,2,3\}$, the admissible word $w$ is said to be \textit{cyclically reduced}. Two admissible words are said to be \textit{equivalent} if they differ by a cyclic permutation on letters. For any admissible word $w$, we denote by $[w]$ its equivalence class.

Let $W$ denote the collection of equivalence classes of all admissible words of letters in $\{\xi_1^{\pm},\xi_2^{\pm},\xi_3^{\pm}\}$. For any oriented curve $\eta$, we denote by $w(\eta)$ one of its associated admissible words. There is thus a one-to-one map from the set of homotopy classes of oriented curves on $P$ to $W$, which sends $[\eta]$ to $[w(\eta)]$.

\medskip
	
When $P$ is equipped a hyperbolic metric, the above discussion can be realized geometrically. Let $X$ be a hyperbolic metric on $P$ with geodesic boundary (it is convenient to think of all three cuffs are realized by non-zero length simple closed geodesics, although this same discussion works when boundary curves are realized by cusps). We start by realizing $\xi_1$, $\xi_2$ and $\xi_3$ as geodesic arcs orthogonal to the boundary geodesics on both ends. Both $H_+$ and $H_-$ are then realized as isometric geodesic hexagons. By reflecting along $\xi_1$, $\xi_2$ and $\xi_3$, we have an orientation reversing isometry $\iota$ of $P$ exchanging $H_+$ and $H_-$. Notice that such an isometry appears for any choice of hyperbolic metric $X$ on $P$.

Let $\eta$ be an oriented closed geodesic on $P$. It intersects $\xi_1\cup\xi_2\cup\xi_3$ transversely. By cutting at the intersection points, the closed geodesic $\eta$ is decomposed into oriented geodesic segments. We choose one intersection point as the starting point, and denote all these geodesic segments by $e_1,...,e_n$ labeled following the orientation of $\eta$. This induces a word $w(\eta)$ of letters in $\{\xi_1^{\pm},\xi_2^{\pm},\xi_3^{\pm}\}$ which is admissible.

\subsection{Curves whose length order never changes}
In this part, we will use the above combinatorial description to prove the Theorem \ref{thm:orderinvariant}.

To do this, we consider the presentation $\pi_1(P)=\langle a,b\rangle$ given previously, and try to translate words of letters in $\{a^{\pm1},b^{\pm1}\}$ to words of letters in $\{\xi_1^{\pm},\xi_2^{\pm},\xi_3^{\pm}\}$. Consider the following correspondence:
    \[
        \begin{aligned}
            \theta(a)&=\xi_1^+\xi_2^-,\\
            \theta(a^{-1})&=\xi_2^+\xi_1^-,\\
            \theta(b)&=\xi_2^+\xi_3^-,\\
            \theta(b^{-1})&=\xi_3^+\xi_2^-.
        \end{aligned}
    \]
Given any reduced word $s_1\cdots s_n$ of letters in $\{a^{\pm1},b^{\pm1}\}$, we consider the word $\theta(s_1)\cdots\theta(s_n)$ of letters in $\{\xi_1^{\pm},\xi_2^{\pm},\xi_3^{\pm}\}$. By taking necessary cancellations between adjacent letters, we have a reduced word $\theta(s_1\cdots s_n)$.
\begin{lemma}\label{lem:techlemma}
    If a reduced word $s_1\cdots s_n$ of letters in $\{a^{\pm1},b^{\pm1}\}$ is cyclically reduced, then $\theta(s_1\cdots s_n)$ is either a cyclically reduced and admissible word or an admissible word $twt'$ with $w$ cyclically reduced and $\{t,t'\}=\{\xi_i^+,\xi_i^-\}$ for some $i\in\{1,2,3\}$.
\end{lemma}    
\begin{proof}
    Let $s_1\cdots s_n$ be any cyclically reduced word of letters in $\{a^{\pm1},b^{\pm1}\}$. We first claim that, for each $1\le j\le n$, as words of $\{\xi_1^{\pm},\xi_2^{\pm},\xi_3^{\pm}\}$, the two letters in $\theta(s_j)$ cannot be canceled simultaneously, when we get $\theta(s_1\cdots s_n)$ from the concatenation $\theta(s_1)\cdots \theta(s_n)$ by taking necessary cancellations between adjacent letters.
    
    We prove this by an induction on $n\in\N^\ast$. When $n=1$, the claim holds. Now we consider a natural number $N>1$, and assume that the claim holds for $1\le n<N$. Consider the case when $n=N$. Consider $s_1\cdots s_{n-1}$. By the induction, the last letter of $\theta(s_{n-1})$ will remain in $\theta(s_1\cdots s_{n-1})$ as the last letter. If there is no cancellation between $\theta(s_{n-1})\theta(s_n)$, then we have the claim. Otherwise, we have $s_{n-1}s_n=ab$ or $b^{-1}a^{-1}$. By our induction assumption on $s_1\cdots s_{n-1}$, to cancel the first letter of $\theta(s_{n-1})$, we have to use the second letter of $\theta(s_{n-2})$. If $n>2$, then we should have $s_{n-2}=a^{-1}$ or $b$, respectively, which are both impossible. Hence, we conclude the claim. From the claim, we can also conclude that the resulting word $\theta(s_1\cdots s_n)$ is independent of the order of cancellations being applied.
    
    Now we prove the lemma. As a consequence of the above claim, the starting letter in $\theta(s_1\cdots s_n)$ will be the first letter of $\theta(s_1)$, while the last letter of $\theta(s_1\cdots s_n)$ will be the last letter of $\theta(s_n)$. Hence, we only have to study the case when $(s_1,s_n)=(b,a)$ and the case when $(s_1,s_n)=(a^{-1},b^{-1})$. We will prove the lemma for the first case, and the other case can be discussed in a similar way. Notice that to cancel $\xi_3^-$ from $\theta(b)$, we need $s_2=b^{-1}$, and to cancel $\xi_1^+$ from $\theta(a)$, we need $s_{n-1}=a^{-1}$. Both of them are impossible, hence $\theta(s_1\cdots s_n)$ is of the form $\xi_2^+w\xi_2^-$ with $w$ starting and ending with $\xi_3^-$ and $\xi_1^+$ respectively. This shows that $w$ is cyclically reduced.
 \end{proof}

We now consider the map
    \[
        \begin{aligned}
            \Theta: [\pi_1(P)]^\ast&\rightarrow W,\\
                    [s_1\cdots s_n]&\mapsto [\theta(s_1\cdots s_n)]. 
        \end{aligned}
    \]
where $[\pi_1(P)]^\ast$ denote the collection of non-trivial conjugacy classes of elements in $\pi_1(P)$. By identifying the elements on both sides to their corresponding curves on $P$, we can see that $\Theta$ is bijective.

Theorem \ref{thm:orderinvariant} is a consequence of the following two results on words of letters in $\{\xi_1^\pm,\xi_2^\pm,\xi_3^\pm\}$.

\begin{proposition}\label{prop:nocancellation}
    Let $w$ and $w'$ be two admissible words of letters in $\{\xi_1^{\pm},\xi_2^{\pm},\xi_3^{\pm}\}$, such that $w$ is cyclically reduced. If the concatenation $ww'$ is also  cyclically reduced and admissible, then for any $X\in\mathcal{T}(P)$, we have
	\[
	\ell_X([\eta])< \ell_X([\eta'']),
	\]
	where $[\eta]$ and $[\eta'']$ are the curves associated to $w$ and $ww'$ respectively.
\end{proposition}
\begin{proof}
Let $X$ be a hyperbolic metric on $P$. We will show that the inequality holds for $X$-geodesics on $P$. Since $X$ is chosen arbitrarily, this will prove the proposition.

Assume, moreover, that all boundary curves are geodesics and $\xi_1$, $\xi_2$ and $\xi_3$ are geodesic arcs orthogonal to boundary geodesics. As before, we denote by $\iota$ the orientation reversing isometry of $P$ by reflecting $P$ along $\xi_1\cup\xi_2\cup\xi_3$.

Let $w=t_1\cdots t_m$ and $w'=t_1'\cdots t_n'$. The geodesic $\eta$ is decomposed into geodesic segments $e_1,...,e_m$. Without loss of generality, we may assume that $e_1\subset H^+$, hence $e_m\subset H_-$. Consider the concatenation $ww'=t_1\cdots t_mt_1'\cdots t_n'$. Since $ww'$ is also  cyclically reduced and admissible, the geodesic $\eta''$ is decomposed into geodesic segments $e_1'',...,e_{m+n}''$. Notice that the geodesic $e''=e_1''\cup\cdots\cup e_m''$ passes $\xi_1$, $\xi_2$ and $\xi_3$ in the same order and along the same orientations as $e=e_1\cup\cdots\cup e_m$. Following the orientation, the endpoint of $e''$ and $e$ may be on different $\xi_i$ and $\xi_j$. In any case, by connecting the endpoints of $e''$ using geodesics in $H_-$, we have a closed piecewise geodesic path $\eta_0$ which is homotopic to $\eta$. Since $\eta$ is the geodesic representative in the homotopy class $[\eta]$, we have
    \[
        \ell_X(\eta)\le \ell_X(\eta_0).
    \]


Now we consider the geodesic $\eta''$ and the geodesic segments $e_{m+1}'',...,e_{m+n}''$. For any index $m+1\le j\le m+n$, if $e_j''\subset H_-$, we replace $e_j''$ by $\iota(e_j'')$ in $\eta''$ and obtain a new closed piecewise geodesic path denoted by $\eta_0''$. 
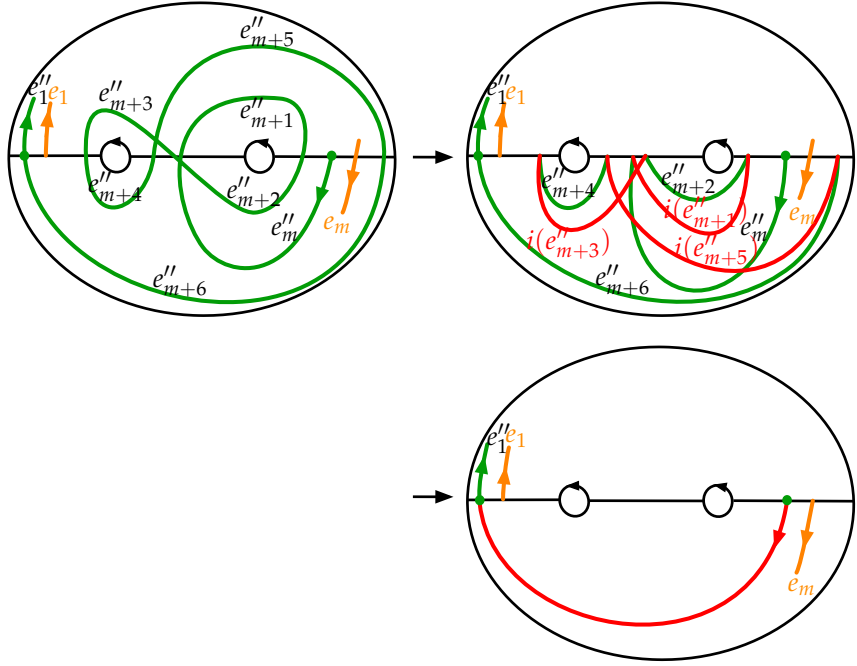
\begin{figure}[h] 
\labellist
\small\hair 2pt   
\pinlabel $e_1''$ at 31 440
\pinlabel $\color{darkorange}{e_1}$ at 45 435
\pinlabel ${\color{darkorange}e_{m}}$ at 252 339
\pinlabel ${e_{m}''}$ at 213 337
\pinlabel $e_{m+1}''$ at 200 422
\pinlabel $e_{m+2}''$ at 190 363
\pinlabel $e_{m+3}''$ at 95 435
\pinlabel $e_{m+4}''$ at 87 367
\pinlabel $e_{m+5}''$ at 200 482
\pinlabel $e_{m+6}''$ at 135 297 

\pinlabel $e_1''$ at 372 440
\pinlabel $\color{darkorange}{e_1}$ at 385 435
\pinlabel ${\color{darkorange}e_{m}}$ at 598 346
\pinlabel ${e_{m}''}$ at 562 337
\pinlabel ${\color{red}i(e_{m+1}'')}$ at 526 348
\pinlabel ${e_{m+2}''}$ at 513 373
\pinlabel ${\color{red}i(e_{m+3}'')}$ at 424 325
\pinlabel ${e_{m+4}''}$ at 424 370
\pinlabel ${\color{red}i(e_{m+5}'')}$ at 534 320
\pinlabel $e_{m+6}''$ at 464 298 

\pinlabel $e_1''$ at 372 180
\pinlabel $\color{darkorange}{e_1}$ at 385 180
\pinlabel ${\color{darkorange}e_{m}}$ at 597 70
\endlabellist
\centering 
\includegraphics[width=12cm]{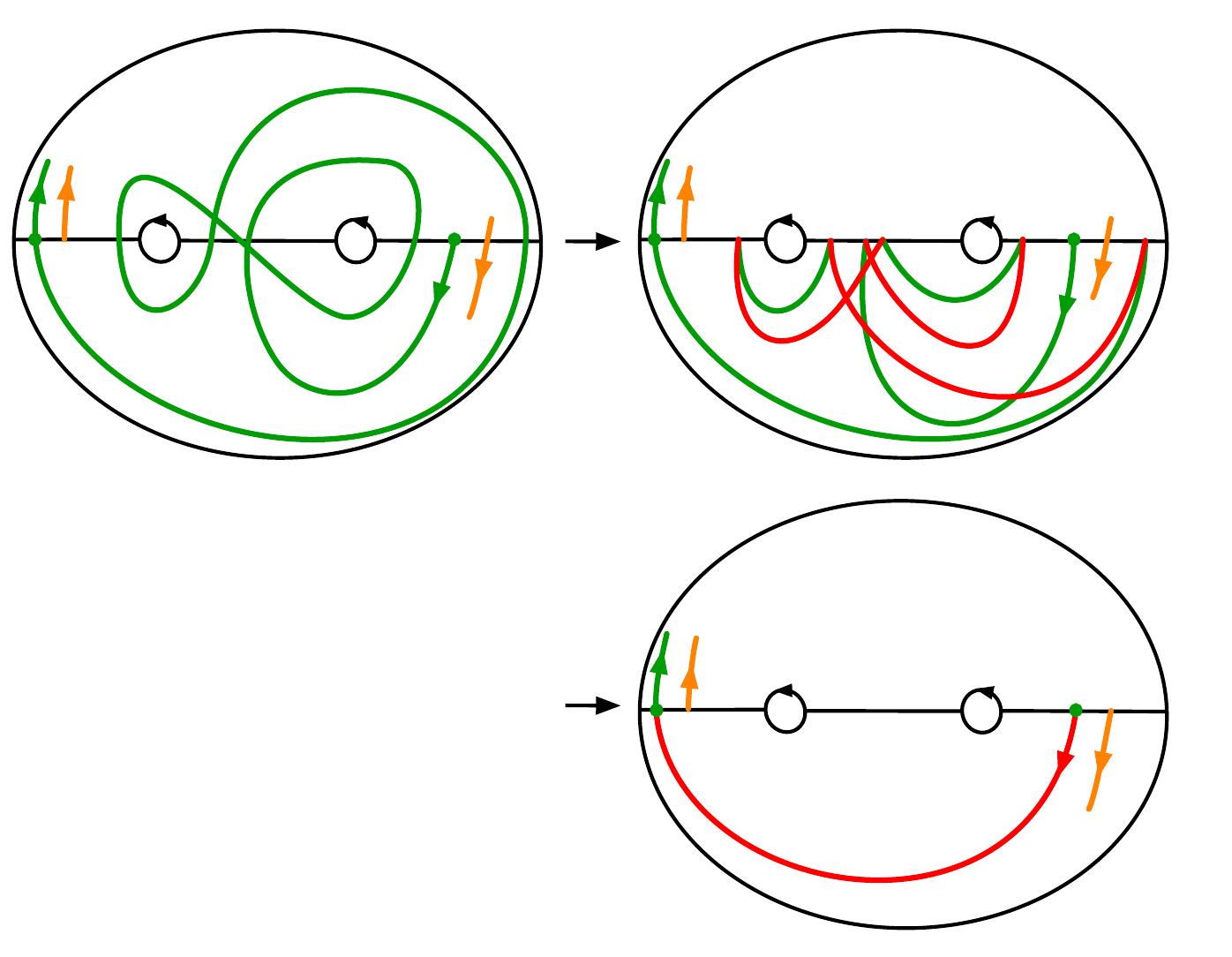}
\caption{Reducing lengths of geodesics via the involution of a pair of pants, I}
\end{figure} 
    
The piecewise geodesic path
    \[
		e_{m}''\cup \iota(e_{m+1}'')\cup e_{m+2}''\cup\cdots\cup \iota(e_{m+n-1}'')\cup e_{m+n}''
	\]
is contained entirely in $H_-$ and connects the starting and ending point of $e$. Hence
    \[
        \ell_X(\eta'')=\ell_X(\eta_0'')> \ell_X(\eta_0)\ge \ell_X(\eta)
    \]
and so
    \[
        \ell_X([\eta''])>\ell_X([\eta]).
    \]
\end{proof}

\begin{proposition}\label{prop:onecancellation}
    Let $w=t_1\cdots t_m$ and $w'=t_1'\cdots t_n'$ be two admissible words of letters in $\{\xi_1^{\pm},\xi_2^{\pm},\xi_3^{\pm}\}$, such that $w$ is cyclically reduced. If $\{t_m,t_1'\}=\{\xi_i^+,\xi_i^-\}$ for some $i\in\{1,2,3\}$, $t_1\neq t_{m-1}$ and $t_1\cdots t_{m-1}t_2'\cdots t_n'$ is  cyclically reduced and admissible, then for any $X\in\mathcal{T}(P)$, we have
	\[
	\ell_X([\eta])< \ell_X([\eta'']),
	\]
	where $[\eta]$ and $[\eta'']$ are curves induced by $w$  and $ww'$ respectively.
\end{proposition}
\begin{proof}
    Let $X$ be a hyperbolic metric on $P$. We will show that the inequality holds for $X$-geodesics on $P$. Since $X$ is chosen arbitrarily, this will prove the proposition.

    Assume moreover that all boundary curves are geodesics and $\xi_1$, $\xi_2$ and $\xi_3$ are geodesic arcs orthogonal to boundary geodesics. Again, let $\iota$ be the orientation reversing isometry of $P$ obtained by reflecting $P$ along $\xi_1\cup\xi_2\cup\xi_3$. Assume that the geodesic $\eta$ is decomposed into geodesic segments $e_1,...,e_m$ and the geodesic $\eta''$ is decomposed into geodesic segments $e_1'',...,e_{m-1}'',e_{m}'',...,e_{m+n-2}''$. Without loss of generality, we assume that $t_1=\xi_1^+$. 
    
    First consider the case when $t_{m-1}=\xi_2^+$. Since $w$ is admissible, we have $t_m=\xi_3^-$ and since $t_1\cdots t_{m-1}t_2'\cdots t_n'$ is admissible, there must be an index $j\in\{2,...,n\}$ such that $t_j'\in\{\xi_3^{+},\xi_3^-\}$, for otherwise, we would have $t_2'\cdots t_n'=\xi_1^-\xi_2^+\cdots \xi_1^-$ ending with $\xi_1^-$, which is a contradiction. 

    Now consider $j\in\{2,...,n\}$. If $t_j'=\xi_3^-$ for some $j$ even, then we replace $e_{k}''$ by $\iota(e_k'')$ for any $k$ such that $k<m+j-2$ even or $k>m+j-2$ odd, and get the piecewise geodesic
        \[
            \eta_0''=e_1''\cup e_2''\cup\cdots\cup e_{m-1}''\cup \iota(e_m'')\cup e_{m+1}''\cup\cdots\cup e_{j-1}''\cup e_{j}''\cup \iota(e_{j+1}'')\cup\cdots \cup \iota(e_{n-1}'')\cup e_n''.
        \] 
\begin{figure}[h] 
\labellist
\small\hair 2pt  
\pinlabel $e_1''$ at 31 400 
\pinlabel ${\color{darkorange}e_{m}}$ at 239 321
\pinlabel ${e_{m}''}$ at 250 370
\pinlabel $e_{m-1}''$ at 95 427
\pinlabel $e_{m-2}''$ at 135 330

\pinlabel $e_{m+1}''$ at 200 429
\pinlabel $e_{m+2}''$ at 190 327
\pinlabel $e_{m+3}''$ at 205 390
\pinlabel $e_{m+4}''$ at 135 284

\pinlabel $e_1''$ at 369 400
\pinlabel $e_{m-2}''$ at 457 330
\pinlabel $e_{m-1}''$ at 455 424
\pinlabel $\color{aogreen}i(e_{m}'')$ at 455 394
\pinlabel $e_{m+1}''$ at 551 430
\pinlabel $e_{m+2}''$ at 514 309
\pinlabel $\color{aogreen}i(e_{m+3}'')$ at 525 341
\pinlabel $e_{m+4}''$ at 451 291

\pinlabel $e_1''$ at 369 150
\pinlabel $e_{m-2}''$ at 514 100
\endlabellist
\centering 
\includegraphics[width=12cm]{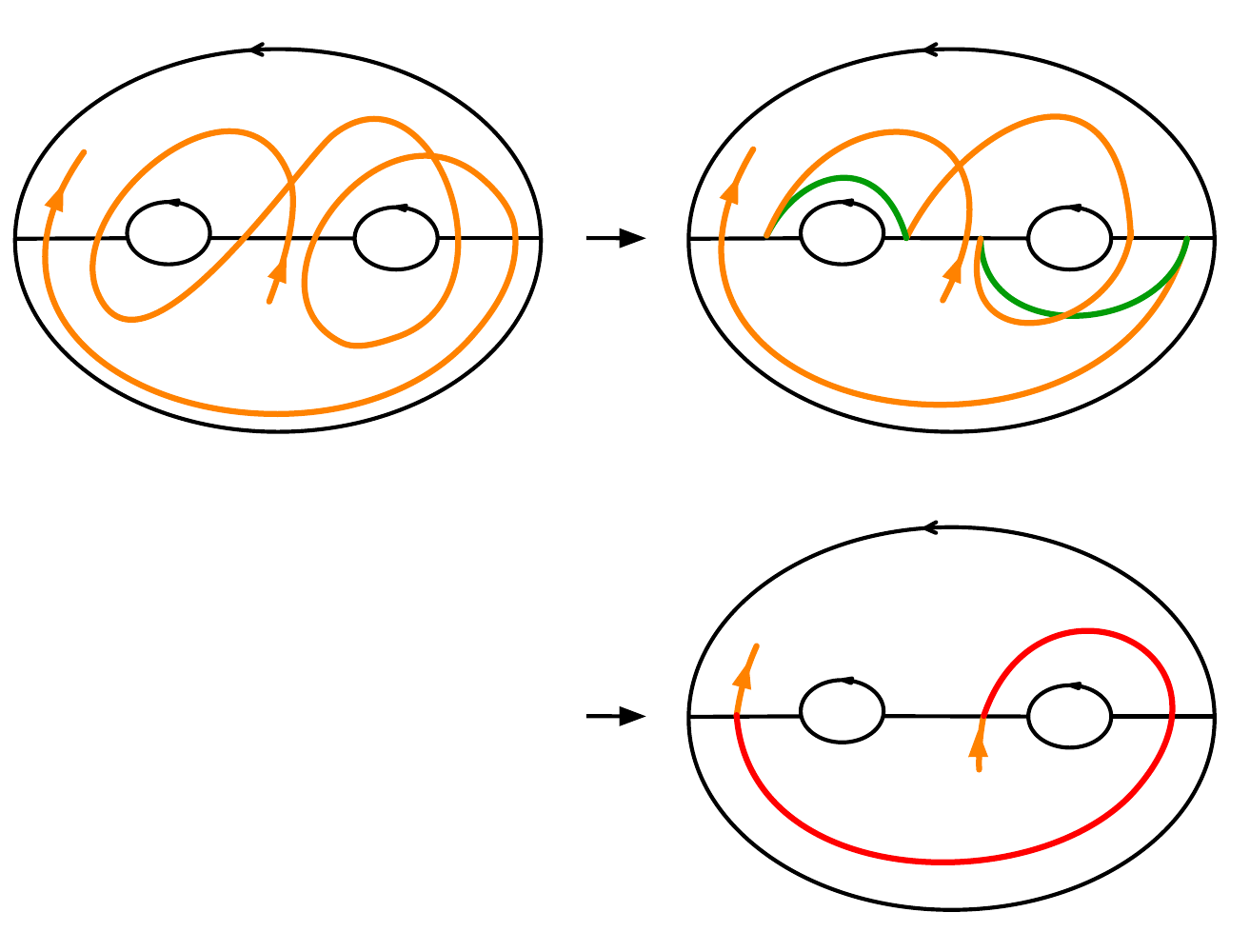}
\caption{Reducing lengths of geodesics via the involution of a pair of pants, II}
\end{figure}
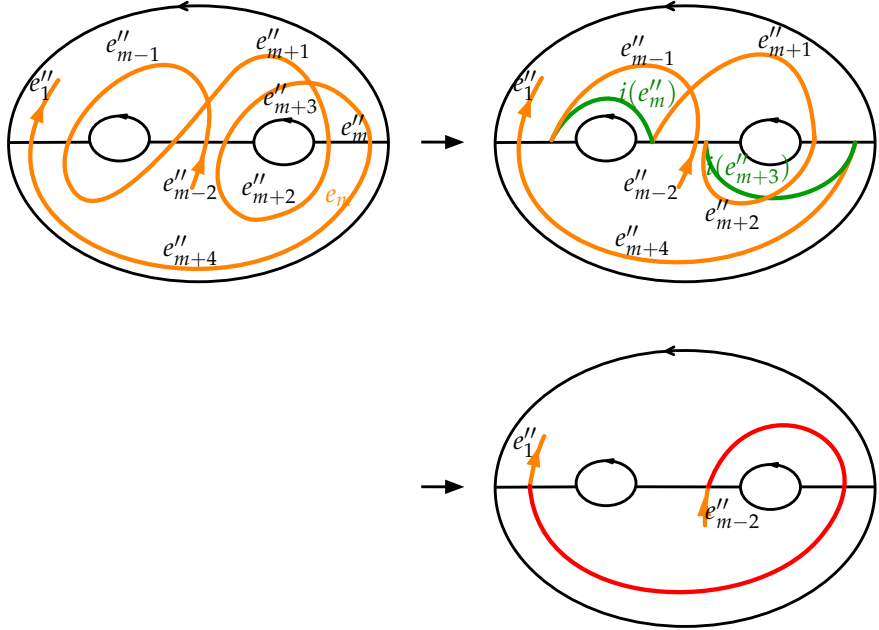 
    
Notice that $e_{m-1}''\cup\iota(e_m'')\cup e_{m+1}''\cup\cdots\cup e_{j-1}''$ is a piecewise geodesic in $H_+$ connecting $\xi_2$ to $\xi_3$, and  $e_{j}''\cup \iota(e_{j+1}'')\cdots \cup \iota(e_{n-1}'')\cup e_n''$ is a piecewise geodesic in $H_-$ connecting $\xi_3$ to $\xi_1$. Hence, $\eta_0''$ is homotopic to $\eta$, but with the same length as $\eta''$. Hence, we have the
        \[
            \ell_X(\eta'')=\ell_X(\eta_0'')>\ell_X(\eta).
        \]

    If $t_j'=\xi_3^+$ for some $j$ odd, then we replace $e_{k}''$ by $\iota(e_k'')$ for any $k$ such that $k\le m+j-2$ even or $k\ge m+j-2$ even, and get the piecewise geodesic
        \[
            \eta_1''=e_1''\cup e_2''\cup\cdots\cup e_{m-1}''\cup \iota(e_m'')\cup e_{m+1}''\cup\cdots\cup \iota(e_{j-1}'')\cup \iota(e_{j}'')\cup e_{j+1}''\cdots \cup \iota(e_{n-1}'')\cup e_n''.
        \]
     Notice that $e_{m-1}''\cup \iota(e_m'')\cup e_{m+1}''\cup\cdots\cup \iota(e_{j-1}'')$ is a piecewise geodesic in $H_+$ connecting $\xi_2$ to $\xi_3$, and  $\iota(e_{j}'')\cup e_{j+1}''\cdots \cup \iota(e_{n-1}'')\cup e_n''$ is a piecewise geodesic in $H_-$ connecting $\xi_3$ to $\xi_1$. Hence, $\eta_1''$ is homotopic to $\eta$ but with the same length as $\eta''$. Hence, we have the
        \[
            \ell_X(\eta'')=\ell_X(\eta_1'')>\ell_X(\eta).
        \]

    The proof in the case when $t_{m-1}=\xi_3^+$ works in the same way.
\end{proof}
Now consider $u$ and $v$, two cyclically words of letters in $\{a^{\pm1}, b^{\pm1}\}$. Moreover, suppose that $u$ starts and ends with different letters. Let $[\eta]$ and $[\eta'']$ be the curves associated to $u$ and $uv$, respectively.

Notice that it is possible to have cancellations in $\theta(u)\theta(v)$, i.e., the last letter of $\theta(u)$ and the first letter of $\theta(v)$ are inverse to each other. This means that we have 
\begin{itemize}
	\item either the last letter of $u$ is $a$ and the first letter of $v$ is $b$;
	\item or the last letter of $u$ is $b^{-1}$ and the first letter of $v$ is $a^{-1}$.
\end{itemize}
In the first case, notice that to cancel $\xi_1^+$ from the left, we need $a^{-1}$ in front of $a$ in $u$ which is impossible. In the second case, then to cancel $\xi_3^+$ from the left, we need $b$ in front of $b^{-1}$ in $u$, which is impossible. Hence, $\theta(u)$ ends with either $\xi_1^+\xi_2^-$ or $\xi_3^+\xi_2^-$ in these two cases respectively. Similarly, we have $\theta(v)$ starting with either $\xi_2^+\xi_3^-$ or $\xi_2^+\xi_1^-$ in these two cases respectively. Therefore, there is only one cancellation.

We now classify all cases according to the starting and ending letters of $u$.

If $u$ starts and ends with $b$ and $a$ respectively, then by Lemma \ref{lem:techlemma}, we have $\theta(u)=twt'$ with $\{t,t'\}=\{\xi_i^+,\xi_i^-\}$ for some $i\in\{1,2,3\}$, and $w$ is a cyclically reduced and admissible word. Hence, $w$ is also associated to $[\eta]$, and $wt'\theta(v)t$ is associated to $[\eta'']$. By applying necessary cancellations to $t'\theta(v)t$, we get a reducible word $w'$. The discussion in the previous paragraph shows that there is no cancellation between $w$ and $w'$. By the proof of Lemma \ref{lem:techlemma}, since $uv$ is cyclically reduced, the last letter of $w'$ is either $t$ or the second last letter in $\theta(v)$. Hence, the word $ww'$ is cyclically reduced. Applying Proposition \ref{prop:nocancellation} to $w$ and $ww'$, we have that Theorem \ref{thm:orderinvariant} holds in this case. 

A similar discussion can be made for the case when $u$ starts and ends with $a^{-1}$ and $b^{-1}$ respectively, and Theorem \ref{thm:orderinvariant} holds for this case as well.

In all other cases, $\theta(u)$ is cyclically reduced. A similar discussion as in the proof of Lemma \ref{lem:techlemma} shows that
	\begin{itemize}
		\item either $\theta(u)=tw$ and $\theta(v)=w't'$ with $\{t,t'\}=\{\xi_i^+,\xi_i^-\}$ for some $i=\{1,2,3\}$, such  that the first letter of $w$ and the last letter of $w'$ are not inverse to each other;
		\item or the first letter of $\theta(u)$ and the last letter of $\theta(v)$ are not inverse to each other.
	\end{itemize}
For the first case, we consider the words $wt$ and $t'w'$. Notice that $wt$ is still cyclically reduced. On the other hand, $ww'$ is reduced. To see this, we may assume, without loss of generality, that $u$ starts with $b$ and $v$ ends with $a$. If $\theta(u)\theta(v)$ is not reduced, we have $u$ ends with $a$ or $b^{-1}$. Both cases are impossible since $\theta(u)$ is cyclically reduced and $u$ is cyclically reduced. Hence, $ww'=\theta(uv)$ is cyclically reduced. Moreover the last letter of $u$ cannot be $b$, since we require $u$ to start and end with different letters. Therefore $w$ cannot start and end with a same letter. Applying Proposition \ref{prop:onecancellation}, we prove Theorem \ref{thm:orderinvariant} in this case.

For the second case, we have $\theta(uv)$ is cyclically reduced. Applying Proposition \ref{prop:nocancellation} when $\theta(u)\theta(v)$ is reduced, or Proposition \ref{prop:onecancellation} when $\theta(u)\theta(v)$ is not reduced, we have Theorem \ref{thm:orderinvariant} in this case.

\subsection{Shortest closed geodesics on pairs of pants}
We will use Theorem \ref{thm:orderinvariant} to look for the shortest closed geodesic with given self-intersection number on a hyperbolic pair of pants.

We will keep the same notation as before. Let $\alpha_1$, $\alpha_2$ and $\alpha_3$ denote the three peripheral curves of $P$ with orientation induced by an orientation of $P$, and consider the presentation $\pi_1(P)=\langle a,b\rangle$ with $a$, $b$ and $c$ corresponding to, respectively,  $[\alpha_1]$, $[\alpha_2]$ and $[\alpha_3]$. Let $X$ be a hyperbolic metric on $P$. Assume now that the boundary curves $\alpha_1$, $\alpha_2$ and $\alpha_3$ are all geodesics, such that $\ell_X(\alpha_1)\le \ell_X(\alpha_2)\le \ell_X(\alpha_3)$. Again, $\xi_1$, $\xi_2$ and $\xi_3$ will be the three common perpendicular geodesics connecting the distinct boundary geodesics. By cutting along $\xi_1\cup\xi_2\cup\xi_3$, we cut all geodesics into geodesic segments connecting $\xi_i$ and $\xi_j$ with $1\le i\neq j\le 3$.

Let $k$ be a natural number, and $\mathcal{G}_{\ge k}$ be the set of all primitive closed geodesics in $P$ with at least $k$ self-intersections. Let $\eta\in\mathcal{G}_{\ge k}$ be such that
    \[
        \ell_X(\eta)=\min\{\ell_X(\eta')\mid \eta'\in \mathcal{G}_{\ge k} \}.
    \]
We are interested in the self-intersection number of $\eta$, and denote it by $i_{\ge k}$.

Baribaud studied closed geodesics on a hyperbolic pair of pants following a similar method in \cite{Baribaud}. Let $\eta_{n1}^{12}$ be the geodesic in $P$ associated to $[a^nb^{-1}]$. One of Baribaud's results \cite{Baribaud}, using our notation, can be phrased as:
\begin{theorem}
Consider a closed geodesic $\eta$ on $P$ corresponding to a cyclically reduced and admissible word of length at least $2(n+1)$ and made of letters in $\{\xi_1^\pm,\xi_2^\pm,\xi_3^\pm\}$. Then it satisfies $\ell_X(\eta_{n1}^{12})\le\ell_X(\eta)$.
\end{theorem}

For small values of $k$, to have $k$ self-intersections, the curve needs $2(k+1)$ letters in its cyclically reduced and admissible word of letters $\{\xi_1^\pm,\xi_2^\pm,\xi_3^\pm\}$. Therefore, from Baribaud's theorem, we obtain the following corollary.
\begin{corollary}
    For $k=1,2,3,4$, we have $i_{\ge k}=k$.
\end{corollary}

As $k$ becomes larger, it is possible that a primitive closed geodesic induced by an admissible word with word length smaller than $2(k+1)$ has more than $k$ self-intersections. Combining this fact and Baribaud's result, to look for $i_{\ge k}$, we just need to compare the length of $\eta_{n1}^{12}$ with those of geodesics in $\mathcal{G}_{\ge k}$ induced by admissible words with less than $2(k+1)$ letters.

In order to find a procedure for computing this, we use an algorithm due to Despr\'{e} and Lazarus \cite{Despre-Lazarus}. In $\pi_1(P)$, any element can be written as a word of letters in $\{a^{\pm1},b^{\pm1}\}$ and Despr\'{e}'s and Lazarus' algorithm computes the self-intersection of the curve associated to a given word. 

The previous discussion on words of letters in $\{a^{\pm1},b^{\pm1}\}$ and words of letters in $\{\xi_1^\pm,\xi_2^\pm,\xi_3^\pm\}$ help us put this all together. Here is the algorithm for searching for all shortest geodesic candidates in $\mathcal{G}_{\ge k}$.

Given $k\in\N^\ast$, our first step is to list the word $aB^k$ and other words of letters in $\{a^{\pm1},b^{\pm1}\}$ with at least $k$ self-intersections that are made of most $2k$ strings, and denote this collection by $\mathcal{S}_{\ge k}$. The second step is to compare all words in $\mathcal{S}_{\ge k}$, and then obtain the geodesic with at least $k$ self-intersection realizing the minimal length of $\mathcal{G}_{\ge k}$.

Here are the lists of candidates for small $k$s:
\begin{itemize}
	\item For $k=1$, $\mathcal{S}_{\ge 1}=\{ab^{-1}\}$. The curve $ab^{-1}$ has $1$ self-intersection, so the shortest geodesic in $\mathcal{G}_{\ge 1}$ has $1$ self-intersection.
	\item For $k=2$, $\mathcal{S}_{\ge 2}=\{aab^{-1}\}$. The curve $aab^{-1}$ has $2$ self-intersections, so a shortest geodesic in $\mathcal{G}_{\ge 2}$ has $2$ self-intersections.
	\item For $k=3$, $\mathcal{S}_{\ge 3}=\{aaab^{-1}, aba^{-1}b^{-1} \}$. Both these curves have $3$ self-intersections, so a shortest geodesic in $\mathcal{G}_{\ge 3}$ has $3$ self-intersections. 
	\item For $k=4$, $\mathcal{S}_{\ge 4}$ consists of 10 curves and the maximum self-intersection number is 5.
		The set of curves in $\mathcal{S}_{\ge 4}$ with $4$ self-intersections is 
			\[
			\{aaaab^{-1}, aaba^{-1}b^{-1}, aab^{-1}a^{-1}b, abba^{-1}b^{-1}, aba^{-1}b^{-1}b^{-1}, ababa^{-1}b^{-1}, abab^{-1}a^{-1}b\}.
			\]
			The set of curves in $\mathcal{S}_{\ge 4}$ with $5$ self-intersections is 
			\[
			\{aaba^{-1}b,abbab^{-1},abab^{-1}a^{-1}b^{-1}\}.
			\] 
			By a smoothing argument, we can see that, in terms of  hyperbolic lengths, three curves with $5$ self-intersections are all longer than those with $4$ self-intersections. More precisely,
			\[
			 \begin{aligned}
			         \ell_X(aaba^{-1}b)>\ell_X(aaba^{-1}b^{-1}),\\ 
        			\ell_X(abbab^{-1}) > \ell_X(aaba^{-1}b^{-1}), \\
        			\ell_X(abab^{-1}a^{-1}b^{-1}) > \ell_X(abab^{-1}a^{-1}b).
			 \end{aligned}
			\] 

	\end{itemize}
Our code that produces these lists is available at \href{https://github.com/hanhv/small-k-systoles}{https://github.com/hanhv/small-k-systoles}. 
From this, we can see that a shortest geodesic in $\mathcal{G}_{\ge 4}$ has $4$ self-intersections.

\medskip

For $k\ge5$, the cardinality of $\mathcal{S}_{\ge k}$ and the maximum self-intersection number of curves in this set  become larger, which makes it much more difficult to compare their hyperbolic lengths on a case by case basis. For example, for $k=5$, $\mathcal{S}_{\ge 5}$ consists of 66 curves and the maximum self-intersection number is 7, whereas for $k = 6$, $\mathcal{S}_{\ge 6}$ consists of 299 curves and the maximum self-intersection number is $11$.
\begin{figure}[h]
				\labellist
				\small\hair 2pt 
				\pinlabel {$aaba^{-1}b$} at 16 5
				\pinlabel {$aaba^{-1}b^{-1}$} at 160 7
				\pinlabel {$a$} at 40 49
				\pinlabel {$b$} at 95 49
				\pinlabel {$a$} at 182 50
				\pinlabel {$b$} at 236 50
				\endlabellist
				\centering \includegraphics[width=12cm]{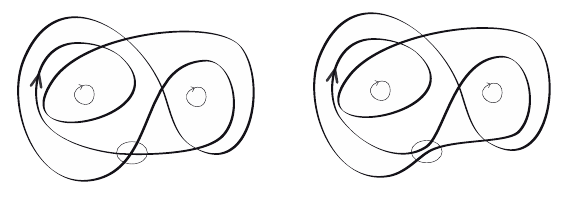}
				\caption{$\ell_X(aaba^{-1}b) > \ell_X(aaba^{-1}b^{-1})$}
				\label{abba^{-1}b}
			\end{figure}
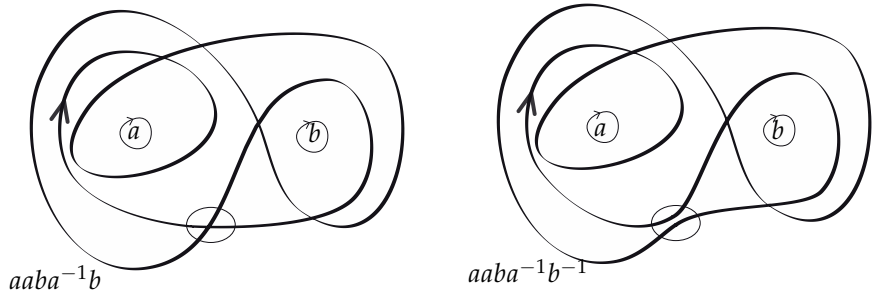   
			
			\begin{figure}[h]
				\labellist
				\small\hair 2pt 
				\pinlabel {$abbab^{-1}$} at 18 14
				\pinlabel {$aaba^{-1}b^{-1}$} at 160 14
				\pinlabel {$a$} at 40 49
				\pinlabel {$b$} at 95 49
				\pinlabel {$a$} at 182 50
				\pinlabel {$b$} at 236 50
				\endlabellist
				\centering \includegraphics[width=12cm]{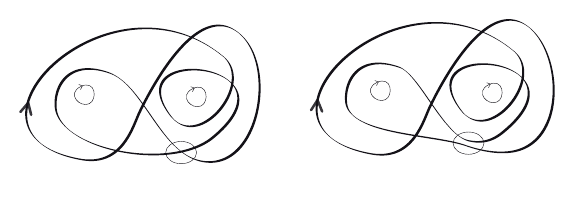}
				\caption{$\ell_X(abbab^{-1}) > \ell_X(aaba^{-1}b^{-1})$}
				\label{abbab^{-1}}
			\end{figure} 
			\begin{figure}[h]
				\labellist
				\small\hair 2pt 
				\pinlabel {$abab^{-1}a^{-1}b^{-1}$} at 21 10
				\pinlabel {$abab^{-1}a^{-1}b$} at 160 10
				\pinlabel {$a$} at 40 49
				\pinlabel {$b$} at 95 49
				\pinlabel {$a$} at 177 45
				\pinlabel {$b$} at 231 45
				\endlabellist
				\centering \includegraphics[width=12cm]{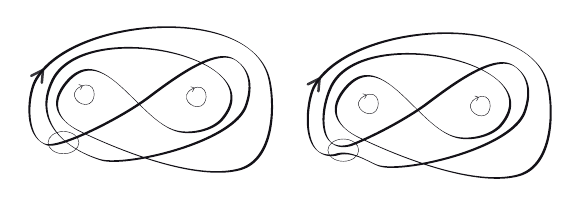}
				\caption{$\ell_X(abab^{-1}a^{-1}b^{-1}) > \ell_X(abab^{-1}a^{-1}b)$}
				\label{abab^{-1}a^{-1}b^{-1}}
\end{figure} 

\begin{remark}
One should notice that, when $k$ is large, we can produce $k$ self-intersection number with $k'$ geodesic segments connecting $\xi_1$, $\xi_2$ and $\xi_3$ in $P$, for $k'<2(k+1)$. By results of the second author \cite{MR4460231}, there are examples where the geodesic realizing the minimal length $\mathcal{G}_{\ge k}$ are completely different for different metrics. Roughly speaking, when the metric is close to a pair of pants with three cusps, for large $k$, the minimum is given by a curve associated to the element $ab^{-k}$. On the other hand, when the boundary lengths are big, the curve with fewer number of geodesic segments in the complement of $\xi_1\cup\xi_2\cup\xi_3$ will have shorter length.
\end{remark}

\bibliographystyle{plain} 
\bibliography{references}

	{\it Addresses:}\\
	Department of Mathematics, University of Fribourg, Fribourg, Switzerland\\
    Department of Mathematics, Purdue University, Indiana, USA\\
	School of Mathematical Sciences, Nankai University, Tianjin, China\\
	{\it Emails:}\\
	hugo.parlier@unifr.ch\\ 
    hanhmfa@gmail.com\\
    binbin.xu.topo@outlook.com

\end{document}